\documentclass[11pt,a4paper]{article}
\usepackage[left=2.5cm, right=2.5cm, top=3cm, bottom=3cm]{geometry}

\usepackage[english]{babel} % 
\usepackage{hyphenat} 
\usepackage{orcidlink}
\usepackage{amsfonts,amsmath,amssymb,amsthm}
\usepackage{graphicx} % Required for inserting images
\usepackage{subfig}
\usepackage{algorithm}
\usepackage{algorithmicx}%
\usepackage{algpseudocode}%
\usepackage{url}
\usepackage{xcolor}
\usepackage{multirow}
\usepackage[titletoc,title]{appendix}
\usepackage{threeparttable} 
\usepackage{authblk}
\usepackage{hyperref}
\providecommand{\keywords}[1]{\textbf{Keywords:} #1}
\usepackage{booktabs}

\usepackage{pifont}
\usepackage{tabularx}     % 提供 \mathfrak
\usepackage[right,mathlines]{lineno}
\theoremstyle{thmstyletwo}
\usepackage[authoryear]{natbib} 
\usepackage{rotating}

\newtheorem{theorem}{Theorem}% meant for sectionwise numbers
\newtheorem{assumption}{Assumption}
\newtheorem{lemma}{Lemma}
\newtheorem{corollary}[theorem]{Corollary}
\newtheorem{remark}{Remark}
\numberwithin{equation}{section}
\usepackage{arydshln}

\title{OFF Newton-type Methods for Unconstrained Optimization with Inexact Derivatives}
\author{Hong~Zhu\thanks{zhuhongmath@126.com}}
\affil{School of Mathematical Sciences, Jiangsu University, Zhenjiang, 212013, Jiangsu, China.}
\date{}

\begin{document}

\maketitle

\begin{abstract}
We study objective-function-free (OFF) Newton-type methods for smooth unconstrained optimization with inexact gradients and Hessians. % 
We propose two methods for first- and second-order stationarity, coupling explicit step-size rules with derivative-accuracy and linear-solver residual conditions to guarantee sufficient decrease without evaluating the objective function. % 
The OFF regularized Newton method finds an $\varepsilon$-first-order stationary point in $\mathcal{O}(\varepsilon^{-2})$ iterations. Under strong convexity, a local Newton variant converges with \(Q\)-order at least $1+\theta$, where \(\theta \in (0, 1]\). %
For second-order stationarity, we combine regularized Newton steps with approximate eigenvectors associated with the minimum eigenvalue. With high probability, the resulting method finds an $(\varepsilon, \sqrt{L_h\varepsilon})$-second-order stationary point in $\mathcal{O}(\varepsilon^{-3/2})$ iterations, using a total of $\widetilde{\mathcal{O}}(\varepsilon^{-7/4})$ approximate gradient evaluations and Hessian-vector products. %
For finite-sum problems, we develop subsampling schemes that yield best-iterate convergence in expectation in the nonconvex first-order setting and local \(R\)-linear convergence with high probability in the strongly convex setting, followed by \(Q\)-order at least \(1+\theta\) once full sampling is reached. The subsampled second-order method achieves a high-probability $\widetilde{\mathcal{O}}(\varepsilon^{-7/2})$ total sample-operation bound, counting component gradient evaluations and component Hessian-vector products.

.%subject classification numbers as needed.
%\subclass{90C26\and 49M15  \and 65K05 }
\end{abstract}

%%%%%%%%%%%%%%%%%%%%%%%%%%%%%%%%%%%%%%%%%    Introduction
%%%%%%%%%%%%%%%%%%%%%%%%%%%%%%%%%%%%%%%%
\section{Introduction}\label{sec:intro}
In this paper, we study objective-function-free (OFF) Newton-type methods for smooth unconstrained optimization problems
\begin{equation}\label{eq:ucfun}
\min_{x\in\mathbb{R}^n} f(x),
\end{equation}
and its finite-sum form 
\begin{equation}\label{eq:fs}
\min_{x\in\mathbb{R}^n}~f(x) = \frac{1}{m}\sum_{i=1}^mf_i(x)
\end{equation}
in the setting where both gradients and Hessians are approximated for first- and second-order stationarity. Given tolerances \(\varepsilon_g, \varepsilon_h > 0\), a point \(x\) is an \(\varepsilon_g\)-first-order stationary point (FOSP) of Problem~\eqref{eq:ucfun} if \(\|\nabla f(x)\| \leq \varepsilon_g\), and an \((\varepsilon_g, \varepsilon_h)\)-second-order stationary point (SOSP) if, in addition, \(\lambda_{\min}(\nabla^2 f(x)) \geq -\varepsilon_h \). Here, \(\|\cdot\|\) denotes the Euclidean norm for vectors and the spectral norm for matrices, and \(\lambda_{\min}(\cdot)\) denotes the minimum eigenvalue of a symmetric matrix. The finite-sum structure in~\eqref{eq:fs} provides an important setting in which inexact derivatives can be obtained through subsampling.

Inexact derivatives introduce additional errors in the predicted decrease, curvature estimates, and stationarity tests. Standard line-search and trust-region acceptance mechanisms use function values to assess a trial step, whereas OFF methods must control these errors through derivative information and other available problem data. This distinction is relevant when function evaluations are costly or unreliable; in finite-sum problems, evaluating the full objective may also undermine the savings obtained by subsampling. Moreover, previous work has shown that convergence guarantees may be more sensitive to errors in objective-function evaluations than to errors in derivative approximations~\cite{CS17,CMS17,BGMS19,BGMT22,BGMT23}. The question addressed in this paper is therefore how to couple derivative-accuracy requirements with explicit step rules so that Newton-type methods retain convergence guarantees without objective-function evaluations.  

Methods using inexact gradients and Hessians have been studied for Problems~\eqref{eq:ucfun} and~\eqref{eq:fs} in~\cite{CGT11a,RM19,XRM20,YXRM21,BDKKR22,BGMT22b,LLHT23,YXRWM23,BGMT25,GJT25b,GJT25c,LW25}. %
These methods differ in whether they require function evaluations, how they control derivative errors, and whether they target first- or second-order stationarity. %  
For strongly convex finite-sum problems, subsampled Newton methods~\cite{RM19} provide local \(R\)-linear convergence, while their global step-size selection may require function evaluations. For nonconvex problems, inexact trust-region, cubic regularization, and line-search methods provide first- or second-order guarantees under various deterministic or probabilistic derivative-accuracy conditions~\cite{CGT11a,XRM20,YXRM21,BDKKR22,BGMT22b,LLHT23,ZX24}. With the coupled tolerances \(\varepsilon_g = \varepsilon\) and \(\varepsilon_h = \sqrt{L_h\varepsilon}\), where \(L_h\) is a Hessian Lipschitz constant, several second-order methods attain an iteration complexity of \(\mathcal{O}(\varepsilon^{-3/2})\)~\cite{NP06,CGT11b,CRS17,GRV18,RW18,ROW20,CRRW21,HHL25}. OFF methods have also been developed for noise-tolerant first-order optimization, stochastic adaptive regularization, and random subspace minimization~\cite{GJT25b,GJT25c,BGMT25}. %
A related line-search-free proximal Newton method was proposed in~\cite{Z25}. When the nonsmooth term is identically zeros, the method reduces to a regularized Newton method using the exact gradient and inexact Hessian. 
For the computation of \((\varepsilon, \sqrt{L_h\varepsilon})\)-SOSP, the inexact Newton-CG (conjugate gradient) method without line-search in~\cite{YXRWM23} achieves an \(\mathcal{O}(\varepsilon^{-3/2})\) iteration complexity with high probability for Problem~\eqref{eq:ucfun} and an \(\mathcal{O}(\varepsilon^{-7/2})\) total sample-operation complexity for Problem~\eqref{eq:fs}. However, its gradient-accuracy requirement depends on quantities that are unavailable before the approximate gradient and search direction are computed. The randomized algorithm without objective-function evaluation~\cite{LW25} finds an \((\varepsilon, \sqrt{L_h\varepsilon})\)-SOSP with high probability in \(\widetilde{\mathcal{O}}(\varepsilon^{-2})\) iterations for Problem~\eqref{eq:ucfun}. For Problem~\eqref{eq:fs}, its total sample-operation is \(\widetilde{\mathcal{O}}(\varepsilon^{-4})\).  Appendix~\ref{appendix:table1} provides a detailed comparison. 

We propose two OFF Newton-type methods for finding first- and second-order stationary points. Both methods control derivative approximation errors and linear-solve residuals to guarantee sufficient objective decrease without function evaluations. %  
For first-order stationarity, the OFF regularized Newton method (OFF-RNM) provides a global first-order complexity bound. A local variant admits fast convergence under strong convexity. When exact gradients and a gradient Lipschitz constant are available, a lazy-gradient rule allows gradient reuse while ensuring the required gradient accuracy. For second-order stationarity, the OFF regularized Newton and negative curvature method (OFF-RN2CM)  supplements regularized Newton steps with approximate minimum-eigenvalue directions to obtain sufficient decrease guarantees without function evaluation. %
We then extend these analyses to finite-sum problems by constructing the required derivative approximations through subsampling. The subsampling variant of OFF-RNM yields \(\lim_{K\to\infty}\min_{0 \leq k\leq K}\mathbb{E}[\|\nabla f(x_k)\|^2] = 0\) in the nonconvex setting, and achieves local \(R\)-linear convergence with high probability in the strongly convex setting. Once full sampling is reached, the iterates converge with \(Q\)-order at least \(1 + \theta\) on the same event. The subsampled OFF-RN2CM finds an \((\varepsilon, \sqrt{L_h\varepsilon})\)-SOSP with high probability in \(\mathcal{O}(\varepsilon^{-3/2})\) iterations and \(\widetilde{\mathcal{O}}(\varepsilon^{-7/2})\) total sample operations. 
 The main results of this paper are summarized in Table~\ref{tab:summary}.

\begin{sidewaystable*}[p] 

{\small
\caption{Summary of the main results. The symbol``\(\sim\)" indicates a result that follows from the analysis but is not proved separately. Algorithm~\ref{alg:rnm} provides a first-order stationary guarantee; Algorithms~\ref{alg:rnmsc} and~\ref{alg:rnmscfs} achieve fast local convergence; and Algorithm~\ref{alg:rnmso} returns an approximate second-order stationary point.}\label{tab:summary}
\setlength{\tabcolsep}{1pt}
%\hspace*{-1cm} %
\begin{tabular*}{\textwidth}{c|c|ccccc} \hline\hline
 {Algs.} &   {Probs.}     & {Approx. gradient (\(g_k\))} &  {Approx. Hessian (\(Q_k\)) }&  {L.c.} &  {Global} &   {Local}                                                                                                                                                                          \\ \hline\hline
 
\multirow{2}{*}{\ref{alg:rnm} } & \multirow{2}{*}{\eqref{eq:ucfun}} & \multirow{2}{*}{\( \|g_k - \nabla f(x_k)\| \leq \bar{\delta}^g_k\|g_k\|\)} & \multirow{3}{*}{\(\|Q_k - \nabla^2 f(x_k)\| \leq \bar{\delta}^h_k\)}  & \multirow{2}{*}{\(\checkmark\)} &Th.~\ref{th:dxksumablecf}  &   \\
 &  &  &   & &\(\mathcal{O}(\varepsilon^{-2})\)  &   \\ %\cdashline{3-4} \cdashline{6-6}  %% 
 
 \ref{alg:pmmlazyg}  &   & \(g_k = \nabla f(z)\) &    &\(\checkmark\)& \(\sim\) & \\ \hdashline %% alg.3
 
\multirow{2}{*}{\ref{alg:rnm} }  &\multirow{2}{*}{\eqref{eq:fs}} & \multirow{2}{*}{\( \mathbb{E}[\|g_k - \nabla f(x_k)\|^2 \mid \mathcal{F}_k] \leq (\bar{\delta}^g_k)^2\varepsilon_k^2\)}   & \multirow{2}{*}{\footnote{Subsampling with \(\mathcal{S}_{h,k} \subseteq [m]\) and \(\vert \mathcal{S}_{h,k}\vert \leq m\).}} & & Th.~\ref{th:rnmfsiter} (\(\mathbb{E}\))    & \\ 
& &  & & &  \footnote{\(\lim_{k\to+\infty}\varepsilon_k = 0\), \(\lim_{K\to+\infty}\min_{0 \leq k < K}\mathbb{E}[\|\nabla f(x_k)\|^2] \to 0\)}   & \\    \cdashline{2-7}  %% alg.3 

\multirow{2}{*}{\ref{alg:rnmsc} } & \multirow{2}{*}{\eqref{eq:ucfun}}& \multirow{2}{*}{\(\|g_k - \nabla f(x_k)\| \leq \hat{\delta}^g_k\|g_k\|\)\footnote{\(\hat{\delta}^g_k \in [0, \min\{\bar{\delta}^g, \|g_k\|^{\theta}, \frac{\sigma - \hat{\delta}^h_k}{2(L_g + \hat{\delta}^h_k)}\}]\) for some \(\bar{\delta}^g < 1\).}} &\multirow{2}{*}{\(\|Q_k - \nabla^2 f(x_k)\| \leq \hat{\delta}^h_k\)\footnote{\(\hat{\delta}^h_k \in (0, \min\{\bar{\delta}^h, \|g_k\|^{\theta}\}]\) for some \(\bar{\delta} \in (0, \sigma)\).}}  & \multirow{2}{*}{\(\times\)} &    & Th.~\ref{th:cr}  \\
 & &   &   &  &    &  \(Q\)-superlinear/quadratic \\  \cdashline{2-7} 
 
 \multirow{3}{*}{\ref{alg:rnmscfs}} &\multirow{3}{*}{\eqref{eq:fs}}&   \multirow{3}{*}{ \( \mathbb{P}(\|g_k \!-\! \nabla f(x_k)\| \!\leq\! a\rho^k \!\mid\! \mathcal{F}_k) \!\geq\! 1 \!-\! p_{g,k}\) }  &\multirow{3}{*}{ \(\mathbb{P}(\|Q_k \!-\! \nabla^2 f(x_k)\| \!\leq\! b\rho^k \!\mid\! \mathcal{F}_k) \!\geq\! 1 \!-\! p_{h,k}\)}    & \multirow{3}{*}{\(\times\)}  &   & Th.~\ref{th:rnmfsstronglyconvex} (\(\mathbb{P}\)) \\
    & &   &    &   &   &   \(R\)-linear\footnote{\(\sum_{k=0}^\infty (p_{g,k} + p_{h,k}) < 1\), where \(0 < p_{g,k}, p_{h,k} < 1\).}  \\ 
& &   &    &   &   &   \(Q\)-superlinear/quadratic   \\
\hdashline %% alg.5 (3)
  
\multirow{4}{*}{\ref{alg:rnmso}}  & \multirow{2}{*}{\eqref{eq:ucfun}} &\multirow{2}{*}{\( \|g_k - \nabla f(x_k)\| \leq \frac{1}{3}\varepsilon \)} &\multirow{2}{*}{\( \|Q_k - \nabla^2 f(x_k)\| \leq \frac{\sqrt{L_h}}{18} \varepsilon^{1/2} \)} &   \multirow{4}{*}{\(\checkmark\)} &  Th.~\ref{th:conrnmso}     & \\
 &  & &  &   &    \(\mathcal{O}(\varepsilon^{-3/2})\)    & \\  \cdashline{3-4}\cdashline{6-6} % --- alg 6. (2)

 & \multirow{2}{*}{\eqref{eq:fs}} & \multirow{2}{*}{\(\mathbb{P}(\|g_k \!-\! \nabla f(x_k)\| \!\leq\! \frac{1}{3}\varepsilon \!\mid\!\mathcal{F}_k) \!\geq\! 1 \!-\! p_{g,k}\)}   & \multirow{2}{*}{\(\mathbb{P}(\|Q_k \!-\! \nabla^2f(x_k)\| \!\leq\! \frac{\sqrt{L_h}}{18}\varepsilon^{1/2}\!\mid\!\mathcal{F}_k) \geq 1 \!-\! p_{h,k}\)} & & Th.~\ref{th:conrnmsofs2} (\(\mathbb{P}\))  &\\ 
    &  &     &  & & \(\mathcal{O}(\varepsilon^{-3/2})\)\footnote{\(\sum_{k=0}^\infty (p_{g,k} + p_{h,k} + p_{\lambda, k}) < 1\), where \(0 < p_{g,k}, p_{h,k}, p_{\lambda, k} < 1\).}   &\\ \hline\hline % --- alg 6. (3)
\end{tabular*}
}
\end{sidewaystable*}

\noindent\textbf{Organization.} 
The remainder of this paper is organized as follows. Section~\ref{sec:OFFrnm} develops the OFF-RNM method, establishes its global first-order complexity and fast local convergence under strong convexity, and develops finite-sum variants.  Section~\ref{sec:OFFn2c} presents the OFF-RN2CM method and establishes iteration, operation, and sample-operation bounds. Concluding remarks are provided in Section~\ref{sec:con}. 
%%*******************************************************************************%%
%  ********     OFF Newton-type method    ********
%%*******************************************************************************%%
\section{OFF regularized Newton method}\label{sec:OFFrnm}
In this section, we study the OFF-RNM method for Problem~\eqref{eq:ucfun} and its finite-sum variants for Problem~\eqref{eq:fs}. 

\subsection{Smooth unconstrained optimization}\label{sec:offrnm2}

We make the following assumption on Problem~\eqref{eq:ucfun}.  
\begin{assumption}\label{assume:ncp}
The function \(f: \mathbb{R}^n \to \mathbb{R}\) is twice continuously differentiable and bounded below by a finite constant \(f_*\). For any initial point \(x_0\in\mathbb{R}^n\), the following conditions hold. 
\begin{itemize}
\item[(i)] The level set \(\mathcal{L}_f(x_0) = \{x \in \mathbb{R}^n \mid f(x) \leq f(x_0)\}\) is bounded. 
\item[(ii)] \(\nabla f\) is \(L_g\)-Lipschitz continuous on an open set \(\Omega_1\) that contains all line segments \([x_k, x_k + d_k]\), where \(x_k\) and \(d_k\) denote the iterates and search directions generated by the algorithm, respectively. 
\end{itemize}
\end{assumption}

\subsubsection{Nonconvex case}\label{subsec:offrnmnonconv}

Newton-type methods take the form of \(x_{k+1} = x_k + d_k\), where \(d_k\) is an approximate solution of the regularized Newton equation \(H_k d = -g_k\) for some vector \(g_k\) and symmetric matrix \(H_k\). 
With exact derivatives and an exact solution of the Newton equation, this iteration reduces to the classical Newton method. If \(x_k, x_k + d_k \in \Omega_1\), then, under Assumption~\ref{assume:ncp}, we have 
\[
f(x_k + d_k) \leq f(x_k) + \langle \nabla f(x_k), d_k\rangle + \frac{L_g}{2}\|d_k\|^2.
\]
Let \(r_k = H_kd_k + g_k\). Then the above inequality yields 
\begin{align}\label{eq:dfdknew}
f(x_k) - f(x_k + d_k) \geq& \langle H_kd_k, d_k\rangle - \langle r_k, d_k\rangle - \langle \nabla f(x_k) - g_k, d_k\rangle - \frac{L_g}{2}\|d_k\|^2 \\
\geq&\lambda_{\min}(H_k)\|d_k\|^2 - (\|r_k\| + \|\nabla f(x_k) - g_k\|)\|d_k\| - \frac{L_g}{2}\|d_k\|^2. \nonumber
\end{align}
Hence, sufficient decrease can be guaranteed by controlling the derivative errors and the linear-solver residual. In this work, we use \(g_k\) and symmetric \(Q_k\) to denote the approximate gradient and Hessian of \(f\) at \(x_k\), respectively, and set \(H_k = Q_k + ([-\lambda_{\min}(Q_k)]_+ +  c_k)I_n\) for some \(c_k > 0\). We assume \(g_k\) and \(H_k\) satisfy the following conditions. 
\begin{assumption}\label{assume:gkuc}
The approximate gradient \(g_k\) and Hessian matrix \(Q_k\) at iteration \(k\) satisfy
    \begin{equation*}%\label{eq:assumegkuc}
    \|g_k - \nabla f(x_k)\| \leq \delta^g_k\|g_k\| \quad {\rm and}\quad \|Q_k - \nabla^2 f(x_k)\| \leq \delta^h_k,
\end{equation*}
where \(\delta^g_k \in [0, \bar{\delta}^g]\) with \(\bar{\delta}^g < \frac{1}{2}\), and \(\delta^h_k \in [0, \bar{\delta}^h]\) with \(\bar{\delta}^h \geq 0\). Moreover,  \(Q_k\) is symmetric. 
\end{assumption}

\begin{algorithm*}[h!]
\caption{OFF regularized Newton method (OFF-RNM).}\label{alg:rnm}
\begin{algorithmic}[1]
\Require{\(\varepsilon > 0\), \(L_g\), \(x_0\in \mathbb{R}^n\), sequence \(\{\eta_k\}\subseteq [0, 1)\), \(\gamma > 1\), \(\bar{\delta}^g \in [0, \frac{1}{2})\), \(\{\delta^g_k\}\subseteq [0, \bar{\delta}^g]\), \(\bar{\delta}^h > 0\), and \(\{\delta^h_k\}\subseteq [0, \bar{\delta}^h]\).}
\For{\(k = 0, 1, \ldots, \)}
\State{Compute approximate gradient \(g_k\);}
\If{\(\|g_k\| \leq \varepsilon\)}
\State{Terminate and return \(x_k\);}
\EndIf
\State{Compute approximate Hessian \(Q_k\);}
\State{Set \(c_k = \gamma\frac{L_g + \eta_k + 2\delta^g_k(\frac{\eta_k}{2} + 2L_g + 2\delta^h_k)}{1 - 2\delta^g_k}\) and define \(H_k = Q_k + ([-\lambda_{\min}(Q_k)]_+ + c_k)I_n\);}
\State{Compute \(x_{k+1} =  x_k + d_k\), where the search direction \(d_k\) satisfies}
\begin{equation}\label{eq:dk}
\|H_k d_k + g_k\| \leq \frac{\eta_k}{2}\|d_k\|. 
\end{equation}
\EndFor\\
\Return{\(\{x_k\}\)}
\end{algorithmic}
\end{algorithm*}

OFF-RNM for Problem~\eqref{eq:ucfun} is summarized in Algorithm~\ref{alg:rnm}. Condition~\eqref{eq:dk} can be enforced by terminating the CG method when the residual satisfies the prescribed bound.  
\begin{lemma}\label{lem:cg}
Suppose that Assumptions~\ref{th:dxksumablecf} and~\ref{assume:gkuc} are satisfied and the CG method is employed to solve the linear system \(H_kd = -g_k\). Let \(\{d_{k,j}\}\) denote the sequence generated by the CG iteration and define \(r_{k,j} := H_kd_{k,j} + g_k\). Then the iterate \(d_{k,j_*}\) satisfies condition~\eqref{eq:dk} provided that \(\|r_{k,j_*}\| \leq \varepsilon^{cg}_k\|g_k\|\), where \(\varepsilon^{cg}_k = \frac{\eta_k/2}{2L_g + 2\delta^h_k + c_k + \eta_k/2}\). 
\end{lemma}
\begin{proof}
Under Assumptions~\ref{th:dxksumablecf} and~\ref{assume:gkuc}, we have \(\|H_k\| \leq 2L_g + 2\delta_k^h + c_k\). Hence,   
\begin{align*}
\|r_{k,j_*}\| \leq& \varepsilon^{cg}_k\|g_k\| \leq \varepsilon^{cg}_k(\|H_kd_{k,j_*}\| + \|r_{k,j_*}\|) \leq \varepsilon^{cg}_k((2L_g + 2\delta^h_k + c_k)\|d_{k,j_*}\| + \|r_{k,j_*}\|),
\end{align*}
which implies
\[
\|r_{k,j_*}\| \leq \frac{\varepsilon^{cg}_k}{1 - \varepsilon^{cg}_k}(2L_g + 2\delta^h_k + c_k)\|d_{k,j_*}\| = \frac{\eta_k}{2}\|d_{k,j_*}\|. 
\] 
The desired result holds. 
\end{proof}

Since \(\|g_k\| \leq \|H_kd_k + g_k\| + \|H_kd_k\| \), condition~\eqref{eq:dk} implies
 \begin{equation}\label{eq:uppergknewton}
 \|g_k\| \leq (\frac{\eta_k}{2} + \|H_k\|)\|d_k\|, \quad \forall k \in \mathbb{N}. 
 \end{equation}
 
Under Assumption~\ref{assume:ncp}, \eqref{eq:dfdknew},~\eqref{eq:dk}, and \(H_k \succeq \frac{1}{2}H_k\) give  
\begin{equation}\label{eq:df}
 f(x_k) -f(x_{k+1}) \geq \frac{1}{2}\langle d_k, (H_k - (L_g + \eta_k)I_n)d_k\rangle + \langle g_k \!-\! \nabla f(x_k), d_k\rangle, \quad \forall k.
 \end{equation} 

\begin{theorem}\label{th:dxksumablecf}
    Suppose that Assumptions~\ref{assume:ncp} and~\ref{assume:gkuc} are satisfied. Let \(\{x_k\}\) be the sequence generated by Algorithm~\ref{alg:rnm}. Let \(\tilde{\beta}_0 = 2L_g + 2\bar{\delta}^h + \frac{\gamma(L_g +1 + 2\bar{\delta}^g(\frac{1}{2}+2L_g + 2\bar{\delta}^h))}{1 - 2\bar{\delta}^g}\) and define 
 \[
  \overline{K}_1 = \lceil \frac{2(\frac{1}{2}+\tilde{\beta}_0)^2(f(x_0) - f_*)}{(\gamma - 1)L_g\varepsilon^{2}}\rceil.
 \]
 Then Algorithm~\ref{alg:rnm} terminates after at most \( \overline{K}_1\) iterations, at an iterate \(x_k\) satisfying \(\|g_k\| \leq \varepsilon\). Moreover, this iterate satisfies \(\|\nabla f(x_k)\| \leq \frac{3}{2}\varepsilon\).  
\end{theorem}
\begin{proof}
From the definitions of \(c_k\) and \(\tilde{\beta}_0\), we have \(\|H_k\| \leq 2L_g + 2\delta_k^h + c_k \leq \tilde{\beta}_0\). \eqref{eq:uppergknewton},~\eqref{eq:df}, and \(H_k \succeq c_kI_n\) give
\begin{align}\label{eq:dfkk}
f(x_k) \!-\! f(x_{k+1}) \geq& \frac{1}{2}(c_k - L_g - \eta_k)\|d_k\|^2 - \delta_k^g\|g_k\|\|d_k\| \!\geq\!  \frac{c_k \!-\! L_g \!-\! \eta_k \!-\! 2\delta_k^g(\frac{\eta_k}{2}\!+\!\|H_k\|)}{2(\frac{\eta_k}{2} + \|H_k\|)^2}\|g_k\|^2 \nonumber \\
\geq& \frac{(\gamma - 1)L_g}{2(\frac{1}{2} + \tilde{\beta}_0)^2}\|g_k\|^2.
\end{align}

We first show that \(\|g_k\| \leq \varepsilon\) for some \(k \leq  \overline{K}_1\). Suppose, for contradiction, that \(\|g_k\| > \varepsilon\) for all \(k = 0, 1, \ldots,  \overline{K}_1\). It follows from~\eqref{eq:dfkk} that 
\begin{align*}
f(x_0) - f(x_{ \overline{K}_1 + 1}) =& \sum_{l = 0}^{ \overline{K}_1}f(x_l) - f(x_{l+1}) \geq \frac{(\gamma - 1)L_g}{2(\frac{1}{2} + \tilde{\beta}_0)^2}\sum_{l=0}^{ \overline{K}_1}\|g_l\|^2 \geq \frac{(\gamma - 1)L_g( \overline{K}_1 + 1)\varepsilon^2}{2(\frac{1}{2} + \tilde{\beta}_0)^2} \\
>& f(x_0) - f_*, 
\end{align*}
where the last inequality follows from the definition of \( \overline{K}_1\). The above inequality contradicts the definition of \(f_*\). 

Finally, Assumption~\ref{assume:gkuc} yields
\[
\|\nabla f(x_k)\| \leq (1 + \delta^g_k)\|g_k\| \leq \frac{3}{2}\varepsilon.
\]
The desired results hold. 
\end{proof}

\begin{remark}
Algorithm 2 in~\cite{Z25} provides an inexact proximal Newton method without line search when the gradient Lipschitz constant is known. When the nonsmooth term is identically zero, it reduces to a unit-step regularized Newton method with a linear-system residual bounded by a multiple of the step norm. OFF-RNM has a similar structure, but explicitly accommodates errors in both the gradient and Hessian approximations. Its regularization parameter \(c_k\) is chosen in conjunction with the derivative-accuracy bounds and the linear-solver residual tolerance to guarantee sufficient decrease without objective-function evaluations. 
\end{remark}
\begin{remark}\label{remark:linesearchalg3}
The gradient-accuracy condition in Assumption~\ref{assume:gkuc} involves \(\|g_k\|\) on its right-hand side, which may make it difficult to enforce directly. % 
When exact gradients and \(L_g\) are available, the lazy-gradient rule in Algorithm~\ref{alg:pmmlazyg} enforces this condition by refreshing the cached gradient whenever the Lipschitz-based error bound exceeds the prescribed tolerance. %
A freshly computed gradient satisfies the condition trivially. If a cached gradient \(g_k = \nabla f(z)\) is reused, the refresh rule and gradient Lipschitz continuity give \(\|g_k - \nabla f(x_k)\| \leq L_g\|x_k - z\| \leq \delta_k^g\|g_k\|\). Hence, the gradient accuracy condition in Assumption~\ref{assume:gkuc} is satisfied. 

\begin{algorithm*}[h!]
\caption{OFF regularized Newton method with lazy gradient  (OFF-RNM-lg).}\label{alg:pmmlazyg}
\begin{algorithmic}[1]
\Require{\(\varepsilon > 0\), \(L_g\), \(x_0\), \(\{\eta_k\}\subseteq [0, 1)\), \(\gamma > 1\), \(\bar{\delta}^g \in [0, \frac{1}{2})\), \(\{\delta^g_k\}\subseteq [0, \bar{\delta}^g]\), \(\bar{\delta}^h > 0\), \(\{\delta^h_k\}\subseteq [0, \bar{\delta}^h]\), \(g_{-1} = g_0 = \nabla f(x_0)\), and \(z \gets x_0\).}
\For{\(k = 0, 1, \ldots, \)}
\If{\(\|x_k - z\| > \frac{\delta_k^g}{L_g}\|\nabla f(z)\|\)}
\State{\(g_k = \nabla f(x_k)\),~\(z \gets x_k\);}
\Else
\State{\(g_k = g_{k-1}\);}
\EndIf
\If{\(\|g_k\| \leq \varepsilon\)}
\State{Terminate and return \(x_k\);}
\EndIf
\State{Compute \(Q_k\);}
\State{Compute \(c_k = \gamma\frac{L_g + \eta_k + 2\delta^g_k(\frac{\eta_k}{2} + 2L_g + 2\delta^h_k)}{1 - 2\delta^g_k}\) and \(H_k \!=\! Q_k \!+\! ([-\lambda_{\min}(Q_k)]_+ \!+\! c_k)I_n\);}
\State{Compute \(x_{k+1} =  x_k + d_k\), where the direction \(d_k\) satisfies~\eqref{eq:dk}.} 
\EndFor\\
\Return{\(\{x_k\}\)}
\end{algorithmic}
\end{algorithm*}
\end{remark}

\subsubsection{Local convergence for strongly convex problems}\label{subsec:offrnmconv}

In this section, we study the local convergence of OFF-RNM with different parameters. We impose the following assumption on Problem~\eqref{eq:ucfun}. 
\begin{assumption}\label{assume:basicf}
The function \(f: \mathbb{R}^n \to \mathbb{R}\) is twice continuously differentiable and \(\sigma\)-strongly convex, and let \(x^* = \arg\min_xf(x)\). There exists a constant \(\hat{\epsilon}_0 > 0\) such that 
\(\nabla f\) is \(L_g\)-Lipschitz continuous and \(\nabla^2 f\) is \(L_h\)-Lipschitz continuous on an open neighborhood containing \(\mathbb{B}(x^*, \hat{\epsilon}_0)\).
\end{assumption}

Under Assumption~\ref{assume:basicf}, there exists \(U_g > 0\) such that \(\|\nabla f(x)\| \leq U_g\) and \(\|\nabla^2f(x)\| \leq L_g\) for any \(x \in \mathbb{B}(x^*, \hat{\epsilon}_0)\).

\begin{assumption}\label{assum:gq}
The approximate gradient \(g_k\) and Hessian \(Q_k\) satisfy
    \[
    \|g_k - \nabla f(x_k)\| \leq \delta^g_k\|g_k\|  \quad {\rm and}\quad \|Q_k - \nabla^2 f(x_k)\| \leq \delta^h_k,
\]
where \(\delta^g_k \in [0, \min\{\bar{\delta}^g, \|g_k\|^{\theta}\}]\) with \(\bar{\delta}^g < 1\) and \(\theta \in (0, 1]\), and \(\delta^h_k \in [0, \min\{\bar{\delta}^h, \|g_k\|^{\theta}\}]\) with \(\bar{\delta}^h\in(0, \sigma)\). Moreover,  \(Q_k\) is symmetric.
\end{assumption}

Under Assumptions~\ref{assume:basicf} and~\ref{assum:gq}, for any \(x_k \in\mathbb{B}(x^*, \hat{\epsilon}_0)\), we have \(\|Q_k\| \leq L_g + \delta^h_k\). In addition,  
\begin{subequations}\label{eq:assumegqsc}
\begin{align}
    &\|\nabla f(x_k) \| \leq (1 + \delta^g_k)\|g_k\|; \label{eq:dgtgsc}\\
   & \|g_k\| \leq \frac{1}{1 - \delta^g_k}\|\nabla f(x_k)\| \leq \frac{1}{1 - \bar{\delta}^g}U_g; \label{eq:uppergk}\\
   &\lambda_{\min}(Q_k) \geq \lambda_{\min}(\nabla^2f(x_k)) - \delta_k^h \geq \sigma - \delta^h_k > 0. \label{eq:qkl}
\end{align}
\end{subequations}
Inequality \eqref{eq:qkl} indicates that \(Q_k\) is positive definite. Accordingly, we substitute \(Q_k\) for \(H_k\) in Algorithm~\ref{alg:rnm}. The OFF Newton-type method for strongly convex optimization is summarized in Algorithm~\ref{alg:rnmsc}. 

\begin{algorithm*}[h!]
\caption{OFF Newton-type method for strongly convex optimization (OFF-Newton-I).}\label{alg:rnmsc}
\begin{algorithmic}[1]
\Require{\(\sigma\), \(x_0\), \(\theta \in (0, 1]\), \(\bar{\delta}^g \in (0, 1)\), \(\{\delta_k^g\}\), and \(\{\delta_k^h\}\).} 
\For{\(k = 0, 1, \ldots, \)}
\State{Compute approximate gradient \(g_k\);}
\If{\(\|g_k\| = 0\)}
\State{Terminate and return \(x_k\);}
\Else
\State{Set \(\mu_k = \min\{1, \frac{\sigma - \delta^h_k}{4\|g_k\|^{\theta}}\}\);}
\EndIf
\State{Compute approximate Hessian \(Q_k\);}
\State{Compute \(x_{k+1} =  x_k + d_k\), where the search direction \(d_k\) satisfies}
\begin{equation}\label{eq:dksc}
\|Q_k d_k + g_k\| \leq \frac{\mu_k}{2}\|g_k\|^{\theta}\|d_k\|.
\end{equation}
\EndFor\\
\Return{\(\{x_k\}\)}
\end{algorithmic}
\end{algorithm*}

\begin{remark} 
Here we present several remarks on Algorithm~\ref{alg:rnmsc}. 
\begin{itemize}
\item[(a)] The only distinction between~\eqref{eq:dksc} and~\eqref{eq:dk} is that \(\eta_k\) on the right-hand side of~\eqref{eq:dk} is replaced by \(\mu_k\|g_k\|^{\theta}\).   
By the same argument as in Lemma~\ref{lem:cg}, condition~\eqref{eq:dksc} can be satisfied if the CG method is adopted to solve the linear system \(Q_kd = -g_k\) with the stopping criterion \(\|r_{k, j}\| \leq \frac{\mu_k\|g_k\|^{\theta}/2}{2L_g + 2\delta^h_k + \mu_k\|g_k\|^{\theta}/2}\|g_k\|\), where \(r_{k,j} = Q_kd_{k, j} + g_k\), and \(\{d_{k,j}\}_{j\in\mathbb{N}}\)  denotes the sequence generated by the CG method. 

\item[(b)] The lazy-gradient mechanism in Algorithm~\ref{alg:pmmlazyg} also applies here if the refresh threshold is replaced by \(\|x_k - z\| > \frac{\min\{\bar{\delta}^g, \|\nabla f(z)\|^{\theta}\}}{L_g}\|\nabla f(z)\|\), and therefore 
\[
\|g_k - \nabla f(x_k)\| \leq L_g\|z - x_k\| \leq \min\{\bar{\delta}^g, \|\nabla f(z)\|^{\theta}\}\|\nabla f(z)\| = \min\{\bar{\delta}^g, \|g_k\|^{\theta}\}\|g_k\|. 
\]
Hence, the gradient-accuracy condition in Assumption~\ref{assum:gq} holds. 
\end{itemize}
\end{remark}

If \(\|g_k\| =0\) for some \(k\), then by Assumption~\ref{assum:gq} and the strong convexity of \(f\), we have \(\nabla f(x_k) = 0\) and \(x_k = x^*\). This means that Algorithm~\ref{alg:rnmsc} terminates at the minimizer \(x^*\) in finitely many iterations. In the remainder of the local analysis, we consider the case in which the algorithm does not terminate finitely. 

Since \(\|g_k\| \leq \|Q_kd_k + g_k\| + \|Q_kd_k\| \), condition~\eqref{eq:dksc} implies  
\begin{equation}\label{eq:boundtildegksc}
        \|g_k\| \leq (\|Q_k\| + \frac{\mu_k}{2}\|g_k\|^{\theta})\|d_k\|, \quad \forall k \in \mathbb{N}. 
    \end{equation}

\begin{theorem}\label{th:cr}
Suppose Assumptions~\ref{assume:basicf} and~\ref{assum:gq} with \(\delta^g_k \in [0, \min\{\bar{\delta}^g, \|g_k\|^{\theta}, \frac{\sigma - \delta^h_k}{2(L_g + \delta_k^h)}\}]\) for some \(\theta \in (0, 1]\) are satisfied. Define
\[
\varsigma_0 = \frac{2(L_h\hat{\epsilon}_0 + 2L_g)}{\sigma - \bar{\delta}^h},~\hat{\varsigma}_1 = (L_g + \sigma + \frac{U_g^{\theta}}{2(1 - \bar{\delta}^g)^{\theta}} ) + \frac{3}{2},~\varsigma_1 = \frac{\hat{\varsigma}_1}{1 - \bar{\delta}^g}(L_g + \sigma + \frac{U_g^{\theta}}{2(1 - \bar{\delta}^g)^{\theta}})^{\theta}
\]
\[
\hat{\epsilon}_1 = \frac{\hat{\epsilon}_0}{1 + \varsigma_0},~\varsigma_2 = \frac{1 + \bar{\delta}^g}{\sigma}(\frac{L_h\varsigma_0^2\hat{\epsilon}_1^{1- \theta}}{2(1 - \bar{\delta}^g)} + \varsigma_1\varsigma_0^{1 + \theta}),~\hat{\epsilon}_2 = \min\{\hat{\epsilon}_1, (2\varsigma_2)^{-1/\theta}\},~\hat{\epsilon}_3 \in (0, \frac{\hat{\epsilon}_2}{1 + 2\varsigma_0}].
\]
Let \(\{x_k\}\) denote the sequence generated by Algorithm~\ref{alg:rnmsc}. Then, the following statements hold. 
\begin{itemize}
\item[(a)] If \(x_k \in \mathbb{B}(x^*,  \hat{\epsilon}_1)\), then  \(x_{k+1}\in \mathbb{B}(x^*, \hat{\epsilon}_0)\) and 
\begin{equation}\label{eq:dxxstarrnm}
\|x_{k+1}  - x^*\| \leq \varsigma_2\|x_k - x^*\|^{1 + \theta}. 
\end{equation}
\item[(b)] If \(x_0 \in \mathbb{B}(x^*, \hat{\epsilon}_3)\), then \(x_k \in \mathbb{B}(x^*, \hat{\epsilon}_2)\) for all \(k\).
\item[(c)] If \(x_0 \in \mathbb{B}(x^*, \hat{\epsilon}_3)\), then \(\{x_k\}\) either terminates at \(x^*\) in finitely many iterations or converges to \(x^*\) with \(Q\)-order at least \(1 + \theta\). 
\end{itemize}
\end{theorem}
\begin{proof}
(a) We first show that if \(x_k \in \mathbb{B}(x^*, \hat{\epsilon}_0)\), then  
\begin{equation}\label{eq:ndk}
\|d_k\| \leq \varsigma_0\|x_k - x^*\|.
\end{equation}
Let \(r_k = Q_kd_k + g_k\). Under Assumption~\ref{assume:basicf}, we have  
 \begin{align*}
 \|d_k\| =& \|Q_k^{-1}(r_k - g_k)\| \nonumber\\
 \leq &||Q_k^{-1}(\nabla f(x^*) - \nabla f(x_k) + \nabla^2f(x_k)(x_k - x^*))\| +\| Q_k^{-1}(\nabla f(x_k) - g_k)\| \nonumber\\
 & + \|Q_k^{-1}r_k\|+ \|Q_k^{-1}(- \nabla^2f(x_k)(x_k - x^*))\| \nonumber \\
 \leq& \frac{L_h}{2}\|Q_k^{-\!1}\|\|x_k \!-\! x^*\|^2 \!+\! \delta^g_k\|Q_k^{-\!1}\|\|g_k\| \!+\! \frac{\mu_k}{2}\|Q_k^{-\!1}\|\|g_k\|^{\theta}\|d_k\| \!+\! \|Q_k^{-\!1}\nabla^2\!f(x_k)\|\|x_k \!-\! x^*\|.
 \end{align*}
 It follows from~\eqref{eq:qkl} that \(\|Q_k^{-1}\| \leq  \frac{1}{\sigma - \delta^h_k}\), using~\eqref{eq:boundtildegksc} and \(\|Q_k\| \leq L_g + \delta_k^h\), we have  
 \begin{align}\label{eq:dkut}
\|d_k\| \!\leq\! & \frac{L_h}{2(\sigma \!-\! \delta^h_k)}\|x_k \!-\! x^*\|^2 \!+\! \frac{\delta^g_k}{\sigma \!-\! \delta^h_k}(\|Q_k\| \!+\! \frac{\mu_k}{2}\|g_k\|^{\theta})\|d_k\|  \!+\! \frac{\mu_k\|g_k\|^{\theta}}{2(\sigma \!-\! \delta^h_k)}\|d_k\| \!+\! \frac{L_g}{\sigma \!-\! \delta^h_k}\|x_k \!-\! x^*\| \nonumber \\
\! \leq\!& \frac{L_h}{2(\sigma \!-\! \delta^h_k)}\|x_k \!-\! x^*\|^2 \!\!+\!\! \frac{L_g}{\sigma \!-\! \delta^h_k}\|x_k \!-\! x^*\| \!\!+\!\! \frac{ 2\delta^g_k(L_g \!+\! \delta^h_k \!+\! \mu_k\|g_k\|^{\theta}/2) \!+\! \mu_k\|g_k\|^{\theta}}{2(\sigma \!-\! \delta^h_k)}\|d_k\|. 
  \end{align}
Since \(\delta^g_k \in [0, \min\{\bar{\delta}^g, \|g_k\|^{\theta}, \frac{\sigma - \delta^h_k}{2(L_g + \delta_k^h)}\}]\), \(\mu_k = \min\{1, \frac{\sigma - \delta^h_k}{4\|g_k\|^{\theta}}\}\), and \(\delta^g_k < 1\), we obtain  
 \[
1 \!-\! \frac{ 2\delta^g_k(L_g \!+\! \delta^h_k \!+\! \mu_k\|g_k\|^{\theta}/2) \!+\! \mu_k\|g_k\|^{\theta}}{2(\sigma - \delta^h_k)} \!=\! \frac{2(\sigma \!-\! \delta^h_k) \!-\! 2\delta^g_k(L_g \!+\! \delta^h_k) \!-\! \mu_k(1 \!+\! \delta^g_k)\|g_k\|^{\theta}}{2(\sigma - \delta^h_k)} \!\geq\!\frac{1}{4}. 
 \]
 Substituting the above inequality into~\eqref{eq:dkut}, recalling that \(\delta^h_k \leq \bar{\delta}^h\), and using the definition of~\(\varsigma_0\), we obtain~\eqref{eq:ndk}. 

 If \(x_k \in \mathbb{B}(x^*,  \hat{\epsilon}_1)\), then we have \(x_k \in \mathbb{B}(x^*, \hat{\epsilon}_0)\) since \(\hat{\epsilon}_1 \leq \hat{\epsilon}_0\). It follows from~\eqref{eq:ndk} that  
\[
\|x_{k+1} - x^*\| \leq \|x_k - x^*\| + \|d_k\| \leq (1 + \varsigma_0)\|x_k - x^*\| \leq \hat{\epsilon}_0,
\]
which implies \(x_{k+1}\in\mathbb{B}(x^*, \hat{\epsilon}_0)\). 

Since \(f\) is \(\sigma\)-strongly convex under Assumption~\ref{assume:basicf}, we have  
\begin{equation}\label{eq:dxkp1xs0}
\|x_{k+1} - x^*\| \leq \frac{1}{\sigma}\|\nabla f(x_{k+1})\| \leq \frac{1 + \delta^g_{k+1}}{\sigma}\|g_{k+1}\|,
\end{equation}
where the second inequality follows from~\eqref{eq:dgtgsc}. Notice that   
\begin{align*}
\|g_{k+1}\| \leq& \|g_{k+1} - r_k\| + \|r_k\| =  \|g_{k+1} - g_k - Q_k(x_{k+1} - x_k)\| + \|r_k\|\\
\leq & \|\nabla f(x_{k+1}) \!-\! \nabla f(x_k) \!-\! \nabla^2f(x_k)(x_{k+1} \!-\! x_k)\| \!+\! \|g_{k+1} \!-\! \nabla f(x_{k+1})\| +\|g_k - \nabla f(x_k)\| \\
&+ \|(Q_k - \nabla^2f(x_k))(x_{k+1} - x_k)\| + \|r_k\| \\
\leq& \frac{L_h}{2}\|x_{k+1} \!-\! x_k\|^2 \!+\! \delta^g_{k+1}\|g_{k+1}\| \!+\! \delta^g_k\|g_k\| \!+\! (\delta^h_k \!+\! \frac{\mu_k}{2}\|g_k\|^{\theta})\|x_{k+1} \!-\! x_k\|\\
\overset{\eqref{eq:boundtildegksc}}{\leq}& \frac{L_h}{2}\|x_{k+1} \!-\! x_k\|^2 \!+\! \delta^g_{k+1}\|g_{k+1}\|  \!+\! (\delta^g_k(L_g \!+\! \delta^h_k \!+\! \frac{\mu_k}{2}\|g_k\|^{\theta}) \!+\! \delta^h_k \!+\! \frac{\mu_k}{2}\|g_k\|^{\theta})\|x_{k+1} \!-\! x_k\|\\
\leq& \frac{L_h}{2}\|x_{k+1} -x_k\|^2 + \delta^g_{k+1}\|g_{k+1}\| + \hat{\varsigma}_1\|g_k\|^{\theta}\|x_{k+1} - x_k\|,
\end{align*}
where the last inequality holds due to \(\delta^g_k \leq \|g_k\|^{\theta}\), \(\delta^h_k \leq \sigma\), \(\mu_k \leq 1\), and~\eqref{eq:uppergk}. Recalling that \(\delta_{k+1}^g
\leq \bar{\delta}^g\), we have  
\begin{align*}
\|g_{k+1}\| \leq& \frac{L_h}{2(1 - \bar{\delta}^g)}\|x_{k+1} - x_k\|^2 + \frac{\hat{\varsigma}_1}{1 - \bar{\delta}^g}\|g_k\|^{\theta}\|x_{k+1} - x_k\|\\
\overset{\eqref{eq:uppergk}, \eqref{eq:boundtildegksc}}{\leq}& \frac{L_h}{2(1 - \bar{\delta}^g)}\|x_{k+1} - x_k\|^2 + \varsigma_1\|x_{k+1} - x_k\|^{1 + \theta}. 
\end{align*}
Substituting the above inequality into~\eqref{eq:dxkp1xs0} gives  
\begin{align*}
\|x_{k+1} - x^*\| \leq& \frac{1 + \delta^g_{k+1}}{\sigma}(\frac{L_h}{2(1 - \bar{\delta}^g)}\|x_{k+1} - x_k\|^2 + \varsigma_1\|x_{k+1} - x_k\|^{1 + \theta})\\
\overset{\eqref{eq:ndk}}{\leq}&\frac{1 + \delta^g_{k+1}}{\sigma}(\frac{L_h\varsigma_0^2}{2(1 - \bar{\delta}^g)}\|x_k - x^*\|^2 + \varsigma_1\varsigma_0^{1 + \theta}\|x_k - x^*\|^{1 + \theta}) \leq \varsigma_2\|x_k - x^*\|^{1 + \theta}.
\end{align*} 
Hence,~\eqref{eq:dxxstarrnm} holds. 

(b) Notice that \(\hat{\epsilon}_3 < \hat{\epsilon}_2\) and \(\varsigma_2\hat{\epsilon}_2^{\theta} \leq \frac{1}{2}\). Since \(x_0 \in \mathbb{B}(x^*, \hat{\epsilon}_3) \subseteq \mathbb{B}(x^*, \hat{\epsilon}_2)\). Suppose inductively that \(x_k \in \mathbb{B}(x^*, \hat{\epsilon}_2)\), part (a) then gives 
\[
\|x_{k+1} - x^*\| \leq \frac{1}{2}\|x_k - x^*\| \leq \frac{1}{2}\hat{\epsilon}_2.
\]
Thus, \(x_{k+1}\in \mathbb{B}(x^*, \hat{\epsilon}_2)\) and the conclusion follows by induction. 

(c) Finite termination yields \(x^*\) by Assumption~\ref{assum:gq} and strong convexity. Otherwise, part (b) gives 
\[
\|x_{k+1} - x^*\| \leq \frac{1}{2}\|x_k - x^*\| \leq \frac{1}{2^k}\|x_0 - x^*\| \to 0,
\]
and part (a) gives \(Q\)-order at least \(1 + \theta\). 
\end{proof}

%%%%%%%%%%%%%%%%%%%%%%%%%%%%%%%%%%%%%%%%%%%%%%%%%%%
% OFF regularized Newton method for finite-sum problems
%%%%%%%%%%%%%%%%%%%%%%%%%%%%%%%%%%%%%%%%%%%%%%%%%%%
\subsection{Finite-sum optimization}\label{sec:offrnm3}

In this section, we study the OFF regularized Newton method for Problem~\eqref{eq:fs}. 

%%%%%%%%%%%%%%%%%%%%%%%%%%%%
%    nonconvex finite-sum
%%%%%%%%%%%%%%%%%%%%%%%%%%%%%
\subsubsection{Nonconvex case} \label{sec:offrnfs}

{We make the following assumption. 
\begin{assumption}\label{assume:ncpfs}
The set of global minimizers of Problem~\eqref{eq:fs} is nonempty, and \(f_*\) denotes the minimum value.
\begin{itemize}
\item[(i)] For each \(i \in [m]:=\{1, \ldots, m\}\), the function \(f_i: \mathbb{R}^n\to\mathbb{R}\) is twice continuously differentiable, and its gradient \(\nabla f_i\) is \(L_{g_i}\)-Lipschitz continuous.
\item[(ii)] For any initial point \(x_0\in \mathbb{R}^n\), the level set \(\mathcal{L}_{f}(x_0) = \{x\vert f(x) \leq f(x_0)\}\) is bounded. 
\end{itemize}
\end{assumption}

Let \(M = \sum_{i=1}^mL_{g_i}\). It follows from Assumption~\ref{assume:ncpfs}(i) that \(\nabla f\) is \(L_g\)-Lipschitz continuous with \(L_g = \frac{M}{m}\). We impose the following uniform variance bounds on the component gradients and Hessians.   

\begin{assumption}\label{assumption:bv}
There exist \(\sigma_g > 0\) such that for all \(x\in\mathbb{R}^n\), 
\begin{equation*}%\label{eq:vars}
\frac{1}{m}\sum_{i=1}^m\|\nabla f_i(x) - \nabla f(x)\|^2 \leq \sigma_g^2. 
\end{equation*}
\end{assumption}

At each iteration, we sample the gradient and Hessian index sets uniformly without replacement and construct the estimators
\begin{equation}\label{eq:subsample}
g_k \triangleq \frac{1}{\vert \mathcal{S}_{g,k}\vert}\sum_{i\in\mathcal{S}_{g,k}}\nabla f_i(x_k) \quad {\rm and}\quad Q_k \triangleq \frac{1}{\vert \mathcal{S}_{h,k}\vert} \sum_{i\in \mathcal{S}_{h,k}}\nabla^2f_i(x_k),
\end{equation}
where \(\mathcal{S}_{g,k}, \mathcal{S}_{h,k} \subseteq [m]\), \(\lvert \mathcal{S}\rvert\) denotes the cardinality of \(\mathcal{S}\). Let  \(\mathcal{F}_k\) denote the information available before sampling at iteration \(k\), defined by \(\mathcal{F}_k = \sigma(x_0, (\mathcal{S}_{g,j}, \mathcal{S}_{h,j})_{0 \leq j \leq k-1})\). The batch sizes are \(\mathcal{F}_k\)-measurable. Under Assumption~\ref{assumption:bv}, we have 
\begin{equation}\label{eq:eggs}
\mathbb{E}[\|g_k - \nabla f(x_k)\|^2 \mid \mathcal{F}_k] = \frac{m - \lvert \mathcal{S}_{g, k}\rvert}{\lvert \mathcal{S}_{g, k}\rvert(m-1)}(\frac{1}{m}\sum_{i=1}^m\|\nabla f_i(x_k) - \nabla f(x_k)\|^2) \leq\frac{\sigma_g^2}{\lvert \mathcal{S}_{g, k}\rvert}.
\end{equation}
By Jensen's inequality and the variance bounds for sampling without replacement,  
\[
\mathbb{E}[\|g_k - \nabla f(x_k)\| \mid \mathcal{F}_k] \leq \frac{\sigma_g}{\sqrt{\lvert \mathcal{S}_{g, k}\rvert}}. 
\]

\begin{theorem}\label{th:rnmfsiter}
Suppose Assumptions~\ref{assume:ncpfs} and~\ref{assumption:bv} hold. Let a prescribed positive sequence \(\{\varepsilon_k\}\) satisfy \(\varepsilon_k \to 0\) as \(k \to \infty\). At each iteration, construct the estimators \(g_k\) and \(Q_k\) by~\eqref{eq:subsample}, using the following sample sizes  
\[
\vert \mathcal{S}_{g, k}\vert = \left\{\begin{array}{ll}
m, & \delta_k^g = 0,\\
\min\{m, \lceil \frac{\sigma_g^2}{(\delta^g_k)^2\varepsilon_k^2}\rceil\}, & \delta_k^g \in (0, \bar{\delta}^g], 
\end{array}
\right.  \quad {\rm and}\quad 1\leq \vert \mathcal{S}_{h, k}\vert \leq m,
\]
where \(\delta_k^g\) is specified before sampling and is \(\mathcal{F}_k\)-measurable.  
Let \(\{x_k\}\) be generated by Algorithm~\ref{alg:rnm} with the stopping test disabled. Then, 
\[
\lim_{K\to+\infty}\min_{0\leq k\leq K}\mathbb{E}[\|\nabla f(x_k)\|^2] = 0. 
\]
\end{theorem}
\begin{proof}
If \(\lvert \mathcal{S}_{g,k} \rvert = m\), then \(g_k = \nabla f(x_k)\), otherwise, it follows from~\eqref{eq:eggs} that  
\begin{equation}\label{eq:egnf}
\mathbb{E}[\|g_k - \nabla f(x_k)\|^2 \mid \mathcal{F}_k] \leq (\delta_k^g)^2\varepsilon_k^2 \leq (\bar{\delta}^g)^2\varepsilon_k^2.
\end{equation}
Since \(\|\nabla^2f_i(x)\| \leq L_{g_i}\), we have \(\|Q_k\| \leq \frac{1}{\vert \mathcal{S}_{h,k}\vert}\sum_{i\in \mathcal{S}_{h,k}}L_{g_i} \leq M\). Moreover, using \(\eta_k < 1\), \(\delta_{g,k}\leq \bar{\delta}^g\), and \(\delta_{h,k} \leq \bar{\delta}^h\), the definition of \(c_k\) gives \(c_k \leq \gamma\frac{L_g + 1 + 2\bar{\delta}^g(\frac{1}{2} + 2L_g + 2\bar{\delta}^h)}{1 - 2\bar{\delta}^g}\). Hence,  
\[
\|H_k\| \leq 2\|Q_k\| + c_k \leq 2M +  \gamma\frac{L_g + 1 + 2\bar{\delta}^g(\frac{1}{2} + 2L_g + 2\bar{\delta}^h)}{1 - 2\bar{\delta}^g} := \beta_1.
\]
\eqref{eq:df}\footnote{Inequality~\eqref{eq:df} requires only gradient Lipschitz continuity,~\eqref{eq:dk}, and positive semidefiniteness of the model matrix, all of which hold pathwise here.}, \(H_k \succeq c_k I_n\), and \(c_k \geq \gamma(L_g + \eta)\) give
\begin{align*}
f(x_k) - f(x_{k+1}) \geq& \frac{(\gamma - 1)L_g}{2}\|d_k\|^2 - \|g_k - \nabla f(x_k)\|\|d_k\|\\
\geq & \frac{(\gamma - 1)L_g}{4}\|d_k\|^2 - \frac{1}{(\gamma - 1)L_g}\|g_k - \nabla f(x_k)\|^2\\
\geq&  \frac{(\gamma - 1)L_g}{4(\frac{1}{2} + \beta_1)^2}\|g_k\|^2 - \frac{1}{(\gamma - 1)L_g}\|g_k - \nabla f(x_k)\|^2,
\end{align*}
where the second inequality follows from \(uv \leq \frac{a}{4}u^2 + \frac{1}{a}v^2\) for any \(a > 0\); the last inequality follows from~\eqref{eq:uppergknewton} and \(\|H_k\| \leq \beta_1\).  

For any \(1 \leq K \in \mathbb{N}\), taking expectations, summing from \(k = 0\) to \(K-1\), using \(f(x_k) \geq f_*\) and~\eqref{eq:egnf} yield
\[
\frac{(\gamma - 1)L_g}{4(\frac{1}{2} + \beta_1)^2}\sum_{k=0}^{K-1}\mathbb{E}[\|g_k\|^2] \leq f(x_0) - f_* + \frac{(\bar{\delta}^g)^2}{(\gamma - 1)L_g}\sum_{k=0}^{K-1}\varepsilon_k^2.
\]
Dividing by \(K\) gives
\[
\frac{1}{K}\sum_{k=0}^{K-1}\mathbb{E}[\|g_k\|^2] \leq \frac{4(\frac{1}{2} + \beta_1)^2}{(\gamma - 1)L_g}\left[\frac{f(x_0) - f_*}{K} + \frac{(\bar{\delta}^g)^2}{(\gamma - 1)L_g}\cdot \frac{1}{K}\sum_{k=0}^{K-1}\varepsilon_k^2\right].
\]
Notice that \(\|\nabla f(x_k)\|^2 \leq 2\|g_k\|^2 + 2\|g_k - \nabla f(x_k)\|^2\), from~\eqref{eq:egnf}, we have \(\mathbb{E}[\|\nabla f(x_k)\|^2] \leq 2\mathbb{E}[\|g_k\|^2] + 2(\bar{\delta}^g)^2\varepsilon_k^2\). Therefore, 
\begin{align*}
\min_{0 \leq k \leq K}\mathbb{E}[\|\nabla f(x_k)\|^2] \leq& \frac{1}{K}\sum_{k=0}^{K-1}\mathbb{E}[\|\nabla f(x_k)\|^2]\\
 \leq& \frac{8(\frac{1}{2} + \beta_1)^2}{(\gamma - 1)L_g}\left[\frac{f(x_0) - f_*}{K} + \frac{(\bar{\delta}^g)^2}{(\gamma - 1)L_g}\cdot \frac{1}{K}\sum_{k=0}^{K-1}\varepsilon_k^2\right] + 2(\bar{\delta}^g)^2\cdot \frac{1}{K}\sum_{k=0}^{K-1}\varepsilon_k^2.
\end{align*}
It follows from \(\varepsilon_k \to 0\) that \(\lim_{K\to+\infty}\frac{1}{K}\sum_{k=0}^{K-1}\varepsilon_k^2 = 0\), and hence
\[
\lim_{K\to+\infty}\min_{0\leq k\leq K}\mathbb{E}[\|\nabla f(x_k)\|^2] = 0. 
\] 
\end{proof}
 The Hessian batch size is unrestricted in Theorem~\ref{th:rnmfsiter}; the proof uses a uniform bound on the sampled Hessian rather than a prescribed Hessian approximation error.

%%%%%%%%%%%%%%%%%%%%%%%%%%%%%%%%%%%
%
%%%%%%%%%%%%%%%%%%%%%%%%%%%%%%%%%%%
\subsubsection{Local convergence for strongly convex finite-sum problems}\label{sec:offrnfsconv}

The following lemma gives sufficient sample sizes for subsampled gradient and Hessian estimates to meet prescribed accuracy levels. 

%%%%%%%%%%%%%%%%%%%%%%%%%%%%
%    Randomized subsampling
%%%%%%%%%%%%%%%%%%%%%%%%%%%%%
 \begin{lemma}\label{lem:subsampling}
 Let \(\widehat{U}_g, \widehat{U}_h > 0\). For any \(x\in\mathbb{R}^n\) satisfying \(\|\nabla f_i(x)\| \leq \widehat{U}_g\) and \(\|\nabla^2 f_i(x)\| \leq \widehat{U}_h\), \(i = 1, \ldots, m\), any \(\epsilon > 0\), and any \( p_g, p_h \in (0, 1)\), draw \(\mathcal{S}_g\) and \(\mathcal{S}_h\) uniformly without replacement from \([m]\), and define \(g(x)\) and \(Q(x)\) as follows:
 \[
 g(x) \triangleq \frac{1}{\vert \mathcal{S}_g\vert}\sum_{i\in\mathcal{S}_{g}}\nabla f_i(x) \quad {\rm and}\quad Q(x) \triangleq \frac{1}{\vert \mathcal{S}_{h}\vert} \sum_{i\in \mathcal{S}_{h}}\nabla^2f_i(x).
 \]
The following statements hold. 
 \begin{itemize}
 \item[(a)]~(Gradient subsampling) If 
 \begin{equation}\label{eq:subg}
 \lvert \mathcal{S}_g\lvert \geq \min\{m, \lceil \frac{4\widehat{U}_g}{\epsilon} \max\{\frac{\widehat{U}_g}{\epsilon}, 1\}\ln\frac{2(n+1)}{p_g} \rceil\},
 \end{equation}
 Then, \(\mathbb{P}(\|g(x) - \nabla f(x)\| \leq \epsilon) \geq 1 - p_g\). 
 \item[(b)]~(Hessian subsampling) If 
 \begin{equation}\label{eq:subh}
 \lvert \mathcal{S}_h\lvert \geq \min\{m, \lceil \frac{16\widehat{U}_h^2}{\epsilon^2}\ln\frac{2n}{p_h}\rceil\}
 \end{equation}
 Then, \(\mathbb{P}(\|Q(x) - \nabla^2 f(x)\| \leq \epsilon) \geq 1 - p_h\). 
 \end{itemize}
 \begin{proof}
 (a) Since \(\|\nabla f_i(x)\| \leq \widehat{U}_g\), \(i = 1, \ldots, m\), at \(x\),  we have 
 \[
 \frac{1}{m}\sum_{i-1}^m\|\nabla f_i(x) - \nabla f(x)\|^2 = \frac{1}{m}\sum_{i=1}^m\|\nabla f_i(x)\|^2 - \|\nabla f(x)\|^2 \leq \widehat{U}_g^2.
 \] 
 Define \(v_i = \nabla f_i(x) - \nabla f(x)\) and \(X_i = \begin{bmatrix}
 0 & v_i^\top\\
 v_i & 0
 \end{bmatrix}\), \(i = 1, \ldots, m\). Since \(\frac{1}{m}\sum_{i=1}^mv_i = 0\), the matrices \(X_1, \ldots, X_m\) have zero means under uniform sampling. Moreover,  
 \[
\|X_i\|_2 = \|v_i\| \leq 2\widehat{U}_g, \quad \|\mathbb{E}[X_i^2]\|_2 \leq \widehat{U}_g^2. 
 \]
 By Operator-Bernstein inequality~\cite[Theorem 1]{GN10note}, we have 
 \[
 \mathbb{P}(\|g(x) - \nabla f(x)\| > \epsilon) = \mathbb{P}(\frac{1}{\lvert \mathcal{S}_g\rvert}\|\sum_{i\in\mathcal{S}_g}X_i\| > \epsilon) \leq 2(n+1)\exp(-\frac{\lvert \mathcal{S}_g\rvert \epsilon}{4}\min\{\frac{\epsilon}{\widehat{U}_g^2}, \frac{1}{\widehat{U}_g}\}),
 \]
 Hence, \eqref{eq:subg} ensures that \(2(n+1)\exp(-\frac{\lvert \mathcal{S}_g\rvert \epsilon}{4}\min\{\frac{\epsilon}{\widehat{U}_g^2}, \frac{1}{\widehat{U}_g}\}) \leq p_g\), which gives the desired result by noting that if \(\lvert \mathcal{S}_g \rvert = m\), then \(g(x) = \nabla f(x)\) and hence \(\mathbb{P}(\|g(x) - \nabla f(x)\| \leq \epsilon) = 1\). 
 
 (b) If \(\vert \mathcal{S}_{h}\vert \geq \lceil \frac{16\widehat{U}_h^2}{\epsilon^2}\ln\frac{2n}{p_h}\rceil\), the result follows from~\cite[Lemma 16]{XRM20}. If \(\vert \mathcal{S}_{h}\vert = m\), then \(Q(x) = \nabla^2f(x)\) and hence \(\mathbb{P}(\|Q(x) - \nabla^2 f(x)\| \leq \epsilon) = 1\). 
 \end{proof}
 \end{lemma}

To analyze the local convergence of the OFF regularized Newton method for Problem~\eqref{eq:fs}, we impose the following assumption. 
\begin{assumption}\label{assume:basicfs}
The function \(f: \mathbb{R}^n \to \mathbb{R}\) is twice continuously differentiable and \(\sigma\)-strongly convex, and let \(x^* = \arg\min_xf(x)\). There exists a constant \(\tilde{\epsilon}_0 > 0\) such that, for each \(i \in [m]\), \(\nabla f_i\) is \(L_{g_i}\)-Lipschitz continuous and \(\nabla^2 f_i\) is \(L_{h_i}\)-Lipschitz continuous on an open neighborhood containing \(\mathbb{B}(x^*, \tilde{\epsilon}_0)\).
\end{assumption}

Under Assumption~\ref{assume:basicfs}, both \(\nabla f\) and \(\nabla^2f\) are Lipschitz continuous on the same neighborhood with constants \(L_g := \frac{1}{m}\sum_{i=1}^mL_{g_i}\) and \(L_h := \frac{1}{m}\sum_{i=1}^mL_{h_i}\), respectively. Define \(\widehat{U}_g = \max_{i\in[m]}\{\|\nabla f_i(x^*)\| + L_{g_i}\tilde{\epsilon}_0\}\) and \(\widehat{U}_h = \max_{i\in[m]}\{\|\nabla^2 f_i(x^*)\| + L_{h_i}\tilde{\epsilon}_0\}\). For every \(x \in B(x^*, \tilde{\epsilon}_0)\), Assumption~\ref{assume:basicfs} and the definitions above imply \(\|\nabla f_i(x)\| \leq \widehat{U}_g\) and \(\|\nabla^2 f_i(x)\| \leq \widehat{U}_h\). Now consider an iteration \(k\) such that \(x_k \in \mathbb{B}(x^*, \tilde{\epsilon}_0)\) and \(0 < p_{g,k}, p_{h,k} < 1\). Choose the sample sizes as 
\begin{equation}\label{eq:nsubgh}
\begin{aligned}
\vert \mathcal{S}_{g,k}\vert =& \min\{m, \lceil \frac{4\widehat{U}_g}{a\rho^{k}}\max\{\frac{\widehat{U}_g}{a\rho^k}, 1\}\ln(2(n+1)/p_{g,k})\rceil\};\\
\vert \mathcal{S}_{h,k}\vert =& \min\{m, \lceil\frac{16\widehat{U}_h^2\ln(2n/p_{h,k})}{b^2\rho^{2k}}\rceil\},
\end{aligned}
\end{equation}
and construct \(g_k\) and \(Q_k\) as in~\eqref{eq:subsample}. Applying Lemma~\ref{lem:subsampling} at \(x_k\) with \(\epsilon = a\rho^k\) for the gradient estimate and \(\epsilon = b\rho^k\) for the Hessian estimate gives  
 \begin{equation}\label{eq:probfsgh}
 \mathbb{P}(\|g_k \!-\! \nabla f(x_k)\| \!\leq\! a\rho^k \mid \mathcal{F}_k) \!\geq\! 1 \!-\! p_{g,k},~~\mathbb{P}(\|Q_k \!-\! \nabla^2 f(x_k)\| \!\leq\! b\rho^k \mid \mathcal{F}_k) \!\geq\! 1 \!-\! p_{h,k}. 
 \end{equation}

The finite-sum variant of Algorithm~\ref{alg:rnmsc}, using the subsampled gradient and Hessian estimators in~\eqref{eq:subsample}, is presented in Algorithm~\ref{alg:rnmscfs}. 

\begin{algorithm*}[h!]
\caption{OFF Newton-type method for strongly convex finite-sum problems (OFF-RNM-fs).}\label{alg:rnmscfs}
\begin{algorithmic}[1]
\Require{\(\widehat{U}_g\), \(\widehat{U}_h\), \(\sigma\), \(x_0\), \(\theta \in(0, 1]\), \(\rho \in (0, 1)\), \(a > 0\), \(\bar{\delta}^h \in (0, \sigma)\), \(b\in(0, \min\{\bar{\delta}^h, \frac{\sigma\rho}{4 + \rho}\}]\), \(\{p_{g, k}\}\subset (0, 1)\), \(\{p_{h,k}\}\subset (0, 1)\).} 
\For{\(k = 0, 1, \ldots, \)}
%\State{\Comment{subsampling}}
\State{Set \(\vert \mathcal{S}_{g,k}\vert\) and \(\vert \mathcal{S}_{h,k}\vert\) as in~\eqref{eq:nsubgh}. Compute \(g_k\) and \(Q_k\) as in~\eqref{eq:subsample};}
%\State{\Comment{update the iterate}}
\State{Compute \(\mu_k = \min\{1, \frac{\sigma - b\rho^k}{4\|g_k\|^{\theta}}\}\) if \(\|g_k\| \neq 0\), otherwise, \(\mu_k = 1\);}
\State{Compute \(x_{k+1} =  x_k + d_k\), where \(d_k\) satisfies \(\|Q_k d_k + g_k\| \leq \frac{\mu_k}{2}\|g_k\|^{\theta}\|d_k\|\).}
\EndFor\\
\Return{\(\{x_k\}\)}
\end{algorithmic}
\end{algorithm*}

\begin{theorem}\label{th:rnmfsstronglyconvex}
Suppose that Assumption~\ref{assume:basicfs} holds. Define the local radius 
\begin{equation}\label{eq:r}
\tilde{\epsilon}_1 = \min\{\frac{\tilde{\epsilon}_0}{2},\frac{(\sigma - b)\rho}{2L_h}, (\frac{7(\sigma - b)^2\rho}{16(L_g + (\sigma - b)\rho/4)^{1+\theta}})^{1/\theta}\}, 
\end{equation} 
where \(\tilde{\epsilon}_0\) is the radius specified therein.  
Suppose \(a \in (0, \frac{(\sigma - b)\rho \tilde{\epsilon}_1}{8}]\) and \(\sum_{k=0}^{\infty}(p_{g,k}+p_{h,k}) \leq \bar{\delta} < 1\). Let \(\|x_0 - x^*\| \leq \frac{\tilde{\epsilon}_1}{2}\). Then, with probability at least \(1 - \bar{\delta}\), 
the following statements hold simultaneously.
\begin{itemize}
\item[(a)] All iterations generated by Algorithm~\ref{alg:rnmscfs} are well-defined and satisfy
\begin{equation}\label{eq:dxkxs}
\|x_k - x^*\| \leq \frac{\tilde{\epsilon}_1}{2}\rho^k \quad\text{for every }k\geq 0.
\end{equation}
\item[(b)] The iterates generated by Algorithm~\ref{alg:rnmscfs} either reach \(x^*\) in finitely many iterations or converge to \(x^*\) with \(Q\)-order at least \(1 + \theta\). 
\end{itemize}
\end{theorem}
\begin{proof}

(a) We prove~\eqref{eq:dxkxs} by induction on the path where \(\|g_k - \nabla f(x_k)\| \leq a\rho^k\) and \(\|Q_k - \nabla^2f(x_k)\| \leq b\rho^k\) hold. On these paths, strong convexity implies \(Q_k \succeq (\sigma - b\rho^k)I_n\).

It is true at $k=0$. Suppose that $\|x_k - x^*\| \leq \frac{\tilde{\epsilon}_1}{2}\rho^k$. Then \(x_k \in \mathbb{B}(x^*, \tilde{\epsilon}_1)\) since \(\rho \in (0, 1)\). Let \(r_k = Q_k d_k + g_k\). By the definition of $\mu_k$, including the stated convention when $g_k = 0$,
\[
\|r_k\| \leq \min\{\frac{1}{2}\|g_k\|^{\theta}, \frac{\sigma - b\rho^k}{8}\}\|d_k\|\leq \frac{\sigma - b\rho^k}{8}\|d_k\|.
\]
Moreover,  
\[
(\sigma - b\rho^k)\|d_k\| \leq \|Q_kd_k\| \leq \|g_k\| + \|r_k\| \leq \|g_k\| + \frac{\sigma - b\rho^k}{8}\|d_k\|.
\]
Hence, \(\|d_k\| \leq \frac{8}{7(\sigma - b\rho^k)}\|g_k\|\). From~\eqref{eq:probfsgh} and the Lipschitz continuity of \(\nabla f\) over \(\mathbb{B}(x^*, \tilde{\epsilon}_1)\), we have 
\begin{align*}
\|r_k\| \leq& \frac12\|g_k\|^\theta\|d_k\| \leq \frac{4}{7(\sigma - b\rho^k)}\|g_k\|^{1+\theta} \leq \frac{4}{7(\sigma - b\rho^k)}(\|\nabla f(x_k)\| + a\rho^k)^{1+\theta} \\
\leq& \frac{4}{7(\sigma - b\rho^k)}(L_g\|x_k - x^*\| + a\rho^k)^{1+\theta}.  
\end{align*}
Since \(Q_k(x_{k+1} - x^*) = Q_k(x_k - x^*) + r_k - g_k\), we have 
\begin{align}\label{eq:dxkp1xs}
& (\sigma - b\rho^k)\|x_{k+1} - x^*\| \leq  \|Q_k(x_{k+1} - x^*)\| \nonumber \\
\leq& \|(Q_k - \nabla^2f(x_k))(x_k - x^*)\| + \|\nabla^2f(x_k)(x_k - x^*) - \nabla f(x_k)\| + \|\nabla f(x_k) - g_k\| + \|r_k\| \nonumber \\
\leq& b\rho^k\|x_k - x^*\| + \frac{L_h}{2}\|x_k - x^*\|^2 + a\rho^k + \frac{4}{7(\sigma - b\rho^k)}(L_g\|x_k - x^*\| + a\rho^k)^{1+\theta},
\end{align}
where the last inequality follows from~\eqref{eq:probfsgh}, the Lipschitz continuity of \(\nabla^2 f\) over \(\mathbb{B}(x^*, \tilde{\epsilon}_1)\), and the previously established bound on \(\|r_k\|\). Hence,  
\begin{align*}
\|x_{k+1} - x^*\| \leq& \frac{\rho^k}{\sigma - b}(\frac{b\tilde{\epsilon}_1}{2} + \frac{L_h\tilde{\epsilon}_1^2}{8} + a + \frac{4}{7(\sigma - b)}(\frac{\tilde{\epsilon}_1}{2}L_g + a)^{1+\theta})\\
\leq& \frac{\rho^k}{\sigma - b}(\frac{\sigma\rho\tilde{\epsilon}_1}{2(4+\rho)} + \frac{L_h\tilde{\epsilon}_1}{8}\frac{(\sigma - b)\rho}{2L_h} + \frac{(\sigma - b)\rho\tilde{\epsilon}_1}{8} + \frac{\tilde{\epsilon}_1(\sigma - b)\rho}{4\cdot 2^{1+\theta}})\\
=& \frac{\rho^{k+1}\tilde{\epsilon}_1}{\sigma - b}(\frac{\sigma}{2(4 + \rho)} + \frac{\sigma - b}{16} + \frac{\sigma - b}{8} + \frac{\sigma - b}{8\cdot 2^\theta})\\
\leq& \frac{7}{16}\rho^{k+1}\tilde{\epsilon}_1 \leq \frac{\tilde{\epsilon}_1}{2}\rho^{k+1},
\end{align*}
where the second inequality uses \(b \leq \frac{\sigma\rho}{4 + \rho}\), \(a \leq \frac{(\sigma - b)\rho \tilde{\epsilon}_1}{8}\), and \(\tilde{\epsilon}_1 \leq (\frac{7(\sigma - b)^2\rho}{16(L_g + (\sigma - b)\rho/4)^{1+\theta}})^{1/\theta}\), the third inequality holds since \(\frac{\sigma}{2(4 + \rho)} \leq \frac{\sigma - b}{8}\). 

Define
\[
E_k = \{\|g_k - \nabla f(x_k)\| \leq a\rho^k, \|Q_k - \nabla^2f(x_k)\| \leq b\rho^k\}, 
\] 
and let \(A_k = \cap_{j=0}^{k-1}E_j\) for \(k \geq 1\). By~\eqref{eq:probfsgh},
\[
\mathbb{P}(E_k^c\mid \mathcal{F}_k) \leq p_{g,k} + p_{h,k} \quad {\rm on}~A_k. 
\]
According to the preceding induction argument, \(x_k\in \mathbb{B}(x^*, \tilde{\epsilon}_1)\) on \(A_k\). Since \(A_k\) is \(\mathcal{F}_k\)-measurable, it follows that \(\mathbb{P}(A_k\cap E_k^c) \leq p_{g,k} + p_{h,k}\). Let \(A = \cap_{k=0}^\infty E_k\). Then \(A^c = \cup_{k=0}^{\infty}(A_k\cap E_k^c)\) and 
\[
\mathbb{P}(A) \geq 1- \sum_{k=0}^{\infty}\mathbb{P}(A_k \cap E_k^c) \geq 1 - \sum_{k=0}^{\infty}(p_{g,k} + p_{h,k}) \geq 1 - \bar{\delta}. 
\]

(b) For the remainder of the proof, we work on \(A\). By part (a), all iterations are well-defined and satisfy~\eqref{eq:dxkxs} on \(A\). Since \(\frac{4\widehat{U}_g}{a\rho^{k}}\max\{\frac{\widehat{U}_g}{a\rho^k}, 1\}\ln(2(n+1)/p_{g,k}) \!\geq\! \frac{4\widehat{U}_g^2\ln(2(n+1)) }{a^2\rho^{2k}} \!\to\! +\infty\) and \(\frac{16\widehat{U}_h^2\ln(2n/p_{h,k})}{b^2\rho^{2k}} \geq \frac{16\widehat{U}_h^2\ln(2n)}{b^2\rho^{2k}} \to +\infty\). If the iterates do not terminate in finitely many iterations, there exists \( \overline{K}_2 \in \mathbb{N}\) such that 
\[
\lvert \mathcal{S}_{g, k}\rvert = m, \quad g_k = \nabla f(x_k),\quad \lvert \mathcal{S}_{h, k}\rvert = m, \quad  Q_k = \nabla^2f(x_k),~~\forall k \geq  \overline{K}_2.
\]
That is, Algorithm~\ref{alg:rnmscfs} performs the inexact Newton step with unit step size for any \(k \geq  \overline{K}_2\). In this case, we have \(\|r_k\| \leq \frac{\sigma - b\rho^k}{8}\|d_k\| \leq \frac{\sigma}{8}\|d_k\|\). Since \(\nabla^2f(x_k) d_k = r_k - \nabla f(x_k)\) and \(\nabla^2 f(x_k) \succeq \sigma I_n\), we have \(\sigma \|d_k\| \leq \|r_k\| + \|\nabla f(x_k)\|\leq \frac{\sigma}{8}\|d_k\| + \|\nabla f(x_k)\|\), which yields \(\|d_k\| \leq \frac{8}{7\sigma}\|\nabla f(x_k)\|\). Moreover, 
\[
\|r_k\| \leq \frac{1}{2}\|\nabla f(x_k)\|^{\theta}\|d_k\| \leq \frac{4}{7\sigma}\|\nabla f(x_k)\|^{1+\theta} \leq \frac{4L_g^{1+\theta}}{7\sigma}\|x_k - x^*\|^{1+\theta},
\]
where the last inequality follows from the Lipschitz continuity of \(\nabla f(x)\) over \(\mathbb{B}(x^*, \tilde{\epsilon}_1)\). 
Using \(\nabla^2f(x_k)(x_{k+1} - x^*) = \nabla^2f(x_k)(x_k - x^*) - \nabla f(x_k) + r_k\) and \(\nabla^2f(x_k) \succeq \sigma I_n\), we obtain  
\[
\sigma\|x_{k+1} - x^*\| \leq \frac{L_h}{2}\|x_{k} - x^*\|^2 + \|r_k\| \leq (\frac{L_h\tilde{\epsilon}_1^{1-\theta}\rho^{(1-\theta)k}}{2^{2-\theta}} + \frac{4L_g^{1+\theta}}{7\sigma})\|x_k - x^*\|^{1+\theta},
\]
where the first inequality uses the Lipschitz continuity of \(\nabla^2f(x)\) and the second inequality follows from~\eqref{eq:dxkxs}.   
Thus, on event \(A\), the iterates either reach \(x^*\) in finitely many iterations or converge to \(x^*\) with \(Q\)-order at least \(1 + \theta\) and statement (b) follows. 
\end{proof}

\begin{remark}\label{remark:alg5}
We provide several remarks regarding Algorithm~\ref{alg:rnmscfs}. 
\begin{itemize}
\item[(a)] The well-definedness guarantee in Theorem~\ref{th:rnmfsstronglyconvex} holds on the joint accuracy event \(A\), whose probability is at least \(1 - \bar{\delta}\). Outside this event, the subsampled Hessian \(Q_k\) may fail to be positive definite, and the prescribed residual condition need not be attainable. 

\item[(b)]  Possible choices of \((p_{g,k}, p_{h,k})\): i) \(p_{g,k} = p_{h,k} = \frac{\bar{\delta}}{2(k+1)(k+2)}\); ii) \(p_{g,k} = p_{h,k} = \frac{\bar{\delta}}{2\overline{w}_{\alpha}}w_k\), where \(w_k = \frac{1}{(k+1)[\ln(k + e)]^{1+\alpha}}\), \(\overline{w}_{\alpha} = \sum_{j=0}^{\infty}w_j\) and \(\alpha > 0\). 

\item[(c)] Algorithm~\ref{alg:rnmscfs} differs from the subsampled Newton method (SSN)~\cite[Algorithm 5]{RM19} mainly in its inexact linear-system solve criterion (see Table~\ref{table:diffsubsample}) and  sampling scheme. SSN samples with replacement and achieves \(R\)-linear local convergence, whereas  Algorithm~\ref{alg:rnmscfs} eventually attains \(Q\)-superlinear convergence (quadratic when \(\theta = 1\)) once full sampling is reached. 

\renewcommand{\arraystretch}{1.2} 
\begin{table}[h!]
\caption{Comparison of Algorithm~\ref{alg:rnmscfs} and SSN.}\label{table:diffsubsample}
\renewcommand{\arraystretch}{1.5}
\begin{tabular}{c|c|c} \hline\hline
& OFF-RNM-fs (Algorithm~\ref{alg:rnmscfs}) & SSN\\ \hline
\multirow{2}{*}{\(d_k\)} & \multirow{2}{*}{\(\|Q_kd_k + g_k\| \leq \frac{\mu_k}{2}\|g_k\|^{\theta}\|d_k\|\)} &  \(\|Q_kd_k + g_k\| \leq \theta_1\|g_k\|\)\\  
& & \(d_k^\top g_k \leq  -(1 - \theta_2)d_k^\top Q_kd_k\)\\  \cdashline{1-3}% d_k
   
 \(g_k\) & \(\mathbb{P}(\|g_k - \nabla f(x_k)\| \leq a\rho^k) \geq 1 - p_{g,k}\) &   \(\mathbb{P}(\|g_k - \nabla f(x_k)\| \leq \rho^k\varepsilon_2) \geq 1 - \bar{\delta}\)   \\ % g_k
 
 \(Q_k\) & \(\mathbb{P}(\|Q_k - \nabla^2 f(x_k)\| \leq b\rho^k) \geq 1 - p_{h,k}\) &   \(\mathbb{P}(\|Q_k - \nabla^2 f(x_k)\| \leq \frac{\rho_0\sigma}{2 + \rho_0}) \geq 1 - \bar{\delta}\)   \\  % Q_k
 
 \(\vert \mathcal{S}_{g,k}\vert\) & \(\min\{m, \lceil \frac{4\widehat{U}_g}{a\rho^{k}}\max\{\frac{\widehat{U}_g}{a\rho^k}, 1\}\ln(2(n+1)/p_{g,k})\rceil\}\) & \(\frac{\widehat{U}_g^2(1 +\sqrt{8\ln(1/\bar{\delta})})^2}{\rho^{2k}\varepsilon_2^2}\)\\ % \mathcal{S}_{g,k}
 
\(\vert \mathcal{S}_{h,k}\vert\) &   \(\min\{m, \lceil\frac{16\widehat{U}_h^2}{b^2\rho^{2k}}\ln\frac{2n}{p_{h,k}}\rceil\}\) & \(\frac{16(2 + \rho_0)^2\widehat{U}_h^2}{\rho_0^2\sigma^2}\ln\frac{2n}{\bar{\delta}}\)\\ \cdashline{1-3}  % \mathcal{S}_{h,k}

\multirow{2}{*}{convergence}  &    \(\mathbb{P}(\|x_k - x^*\| \leq \frac{\tilde{\epsilon}_1}{2}\rho^k) \geq 1 - \bar{\delta}\); \tnote{a} & \(\mathbb{P}(\|x_{k+1} - x^*\| \leq c\rho^k)\geq (1 - \bar{\delta})^{2k}\) \\ 
& \(\mathbb{P}(Q\text{-superlinear/quadratic}) \geq 1 - \bar{\delta}\). &\\
\hline\hline                                            
\end{tabular}
\begin{tablenotes}
\item[a] \(\sum_{k=0}^{\infty}(p_{g,k}+p_{h,k}) \leq \bar{\delta} < 1\).
\end{tablenotes}
\end{table}
\end{itemize}
\end{remark}

\section{OFF regularized Newton and negative curvature method}\label{sec:OFFn2c}
In this section, we study the OFF-RN2CM method for seeking an \((\varepsilon_g, \varepsilon_h)\)-second-order stationary point up to fixed constant factors with given \(\varepsilon_g, \varepsilon_h > 0\). 

\subsection{Smooth unconstrained optimization}\label{sec:OFFn2cp2}
We make the following assumption on Problem~\eqref{eq:ucfun}. 
\begin{assumption}\label{assume:ncpso}
The function \(f: \mathbb{R}^n \to \mathbb{R}\) is twice continuously differentiable and bounded below by a finite constant \(f_*\). For any initial point \(x_0\in\mathbb{R}^n\), the following conditions hold.
\begin{itemize}
\item[(i)] The level set \(\mathcal{L}_f(x_0) = \{x\vert f(x) \leq f(x_0)\}\) is bounded.
\item[(ii)] \(\nabla f\) and \(\nabla^2 f\) are Lipschitz continuous on an open set \(\Omega_2\) that contains all line segments \([x_k, x_{k+1}]\) with constants \(L_g\) and \(L_h\), where \(x_k\) and \(d_k\) denote the iterates and search directions generated by the algorithm, respectively. 
\end{itemize}
\end{assumption}

Under Assumption~\ref{assume:ncpso}, one has \(\|\nabla^2f(x)\| \leq L_g\) for any \(x\in\Omega_2\). Moreover,  if \(x_k, x_k + \alpha d_k \in \Omega_2\) for some \(\alpha\in(0, 1]\), then, 
\begin{equation}\label{eq:uppercubic}
f(x_k + \alpha d_k) \leq f(x_k) + \alpha \langle \nabla f(x_k), d_k\rangle + \frac{\alpha^2}{2}\langle d_k, \nabla^2f(x_k)d_k\rangle + \frac{L_h\alpha^3}{6}\|d_k\|^3.
\end{equation}
Let \(g_k\) and symmetric \(Q_k\) be the approximate gradient and Hessian of \(f\) at \(x_k\), respectively. \eqref{eq:uppercubic} yields
\begin{align*}
f(x_k) - f(x_k + \alpha d_k) \geq& -\alpha \langle g_k, d_k\rangle -\alpha \langle \nabla f(x_k) - g_k, d_k\rangle - \frac{\alpha^2}{2}\langle d_k, (\nabla^2f(x_k) - Q_k)d_k\rangle \\
&- \frac{\alpha^2}{2}\langle d_k, (Q_k + 2\varepsilon_h I_n)d_k\rangle + \alpha^2\varepsilon_h\|d_k\|^2 - \frac{L_h\alpha^3}{6}\|d_k\|^3 
\end{align*}
Hence, sufficient decrease can be guaranteed by controlling the derivative errors and choosing the direction and step size appropriately.   
We assume \(g_k\) and \(Q_k\) satisfy the following conditions.

\begin{assumption}\label{assume:gkhkso}
The approximate gradient \(g_k\) and Hessian \(Q_k\) at iteration \(k\) satisfy 
\[
    \|g_k - \nabla f(x_k)\| \leq \frac{1}{3}\varepsilon_g \quad {\rm and}\quad \|Q_k - \nabla^2 f(x_k)\| \leq \frac{1}{18} \varepsilon_h,
\]
where \(0 < \varepsilon_g, \varepsilon_h \ll \min\{L_g, L_h\}\). Moreover,  \(Q_k\) is symmetric.
\end{assumption}

OFF-RN2CM is an OFF-type variant of the Newton-CG method in~\cite{ROW20}. The latter uses exact derivatives and a line search to guarantee sufficient decrease. It computes a search direction by approximately solving regularized Newton equation \((\nabla^2f(x_k) + 2\varepsilon_h I_n)d = -\nabla f(x_k)\) when \(\nabla^2f(x_k) + 2\varepsilon_h I_n\succeq 0\). Otherwise, it chooses a negative-curvature direction of \(\nabla^2f(x_k)\). In contrast, OFF-RN2CM uses inexact derivatives and explicit step sizes, together with an approximate regularized Newton direction or an approximate eigenvector associated with the minimum eigenvalue. The estimation accuracy of the approximate eigenvector is characterized by the following assumption.

\begin{assumption}\label{assume:lampso}
Let \(\hat{p}_k\) denote a unit approximate eigenvector of \(Q_k\) associated with its approximate minimum eigenvalue,  and define \(\hat{\lambda}_k:= \hat{p}_k^\top Q_k\hat{p}_k\). We assume that
\[
\hat{\lambda}_k - \lambda_{\min}(Q_k) \leq \frac{1}{2}\varepsilon_h.
\]
\end{assumption}

Assumption~\ref{assume:lampso} can be satisfied with high probability by applying the Lanczos method to \(Q_k\), initialized with a random vector uniformly distributed on the unit sphere; see~\cite[Lemma 10]{ROW20}.

\begin{algorithm*}[h!]
\caption{OFF regularized Newton and negative curvature method (OFF-RN2CM).}\label{alg:rnmso}
\begin{algorithmic}[1]
\Require{\(x_0\), \(L_h\), \(\varepsilon_g\in(0, L_h)\), \(\varepsilon_h = \sqrt{L_h\varepsilon_g}\), \(\hat{\mu} > 0\), and \(\eta > 0\)}
\For{\(k = 0, 1, \ldots\)}
\State{Compute \(g_{k}\), \(Q_k\) and \((\hat{\lambda}_k,\hat{p}_k)\);}
\If{\(\|g_k\| > \varepsilon_g\)}\Comment{\textbf{large-gradient branch}}
\If{\(\hat{\lambda}_k < -\varepsilon_h\)}
\State{\(d_{\rm type} = {\rm NC}\);}\Comment{negative curvature step}\label{line:5}
\Else
\State{Invoke the CG method to solve the regularized Newton system \((Q_k + 2\varepsilon_hI_n)d = -g_k\) with the zero vector as the initial guess, and return a vector \(d_k\) that satisfies}\label{line:7}
\begin{equation}\label{eq:dkso}
   \| (Q_k + 2\varepsilon_hI_n)d_k + g_k \| \leq \hat{\mu}\varepsilon_h\|d_k\|;
\end{equation}
\If{\(\|d_k\| \leq \frac{2\varepsilon_g}{\varepsilon_h}\)}
\State{Terminate and return \(x_{k+1} \gets x_k + d_k\);}\label{line:9}
\Else
\State{\(d_{\rm type} = {\rm SOL}\);}\Comment{approximate Newton step}\label{line:11}
\EndIf
\EndIf
\Else\Comment{\textbf{small-gradient branch}}
\If{\(\hat{\lambda}_k \geq -\varepsilon_h\)}
\State{Terminate and return \(x_k\);}\label{line:17}
\Else
\State{\(d_{\rm type} = {\rm NC}\);}\label{line:19}\Comment{negative curvature step}
\EndIf
\EndIf
\State{\Comment{update the iterate}}
\If{\(d_{\rm type} = {\rm NC}\)}
\State{\(d_k = -{\rm sign}(g_k^\top\hat{p}_k)\hat{p}_k\)\footnotemark and \(\alpha_k = \frac{17}{12L_h}\varepsilon_h\);}\label{line:24}
\Else
\State{\(\alpha_k = \frac{\sqrt{2}}{\sqrt{3(L_h+\eta)}\|d_k\|}\varepsilon_g^{1/2}\);}\label{line:26}
\EndIf
\State{\(x_{k+1} \gets x_k + \alpha_kd_k\);}
\EndFor\\
\Return{\(\{x_k\}\)}
\end{algorithmic}
\end{algorithm*}
\footnotetext{Here, we define \({\rm sign}(t) = 1\) if \(t \geq 0\), \({\rm sign}(t) = -1\) otherwise.}

The OFF-RN2CM algorithm is summarized in Algorithm~\ref{alg:rnmso}. At each iteration, Algorithm~\ref{alg:rnmso} selects either the large-gradient branch or the small-gradient branch according to \(\|g_k\|\). In either branch, detected negative curvature leads to a signed step along the approximate minimum-eigenvalue direction. If no significant negative curvature is detected, the algorithm returns the current point when the approximate gradient is small. Otherwise, it computes an approximate regularized Newton direction, taking a unit step and terminating if the direction is sufficiently small, or taking a damped regularized  Newton step and continuing the iteration otherwise. The step-size rules are designed to guarantee sufficient decrease at every nonterminating iteration without function-value-based line searches. The termination tests ensure approximate first- and second-order stationarity of the returned point. 

\begin{remark}
The main distinction between Algorithm~\ref{alg:rnmso} and the inexact Newton-CG (I-N-CG) method~\cite[Algorithm 4]{YXRWM23} lies in the specification of the negative-curvature direction. Algorithm~\ref{alg:rnmso} employs the unit eigenvector corresponding to the estimated minimum eigenvalue of \(Q_k\) as the negative-curvature direction, instead of an arbitrary negative curvature vector generated by the Capped CG method~\cite{ROW20}. This design simplifies the accuracy requirement for the approximate gradient estimator. Specifically, the I-N-CG method  enforces the condition \(\|g_k - \nabla f(x_k)\| \leq \frac{1 -\zeta}{8}\min\{\frac{3\varepsilon_h^2}{65(L_h + \eta)}, \max\{\varepsilon_g, \min\{\varepsilon_h\|d_k\|, \|g_k\|, \|g_{k+1}\|\}\}\}\) for some \(\zeta, \eta \in (0, 1)\), where \(d_k\) and \(g_{k+1}\) are unknown and depend on \(g_k\) for computation. Our condition is simpler to prescribe before sampling, although its absolute gradient-error tolerance can be more restrictive when the gradient norm is large. 
\end{remark}

\begin{lemma}\label{lem:dfncso}
    Suppose Assumptions~\ref{assume:ncpso},~\ref{assume:gkhkso},  and~\ref{assume:lampso} are satisfied. If, at iteration \(k\) of Algorithm~\ref{alg:rnmso}, \(d_{\rm type} = {\rm NC}\), then we have 
\begin{equation}\label{eq:dkncso}
    f(x_{k+1}) \leq f(x_k) - \frac{17}{10368\sqrt{L_h}}\varepsilon_g^{3/2}.
    \end{equation}
\end{lemma}
\begin{proof}
     Notice that 
    \[
    \vert \hat{p}_k^\top Q_k\hat{p}_k - \hat{p}_k^\top \nabla^2f(x_k)\hat{p}_k\vert \leq \|Q_k -\nabla^2f(x_k)\|\|\hat{p}_k\|^2 \leq \frac{1}{18}\varepsilon_h \leq \frac{1}{18}\vert \hat{\lambda}_k\vert,
    \]
    where the second inequality follows from \(\|\hat{p}_k\| = 1\) and the last inequality follows from \(\hat{\lambda}_k < -\varepsilon_h\) when \(d_{\rm type} = {\rm NC}\). Hence, we have 
    \begin{equation}\label{eq:pn2fp}
    \hat{p}_k^\top \nabla^2f(x_k)\hat{p}_k \leq \frac{1}{18}\vert \hat{\lambda}_k\vert + \hat{p}_k^\top Q_k\hat{p}_k = -\frac{17}{18}\vert \hat{\lambda}_k\vert. 
    \end{equation}
    Therefore, 
    \begin{align*}
    f(x_k \!+\! \alpha_kd_k) \!\overset{\eqref{eq:uppercubic}}{\leq}\!&\! f(x_k) + \alpha_k\nabla f(x_k)^\top d_k + \frac{\alpha_k^2}{2}d_k^\top \nabla^2f(x_k)d_k + \frac{L_h\alpha_k^3}{6}\|d_k\|^3\\
   \! =\!& f(x_k) \!-\! \alpha_k{\rm sign}(g_k^\top\hat{p}_k)\nabla f(x_k)^\top \hat{p}_k \!+\! \frac{\alpha_k^2}{2}\hat{p}_k^\top \nabla^2f(x_k)\hat{p}_k \!+\! \frac{L_h\alpha_k^3}{6}\\
  \!\overset{\eqref{eq:pn2fp}}{\leq}\! &\! f(x_k) \!-\! \alpha_k{\rm sign}(\!g_k^\top\hat{p}_k\!)(\!g_k^\top \hat{p}_k\!) \!-\!\! \frac{17\alpha_k^2}{36}\vert \hat{\lambda}_k\vert \!+ \!\frac{L_h\alpha_k^3}{6}  \!-\! \alpha_k{\rm sign}(\!g_k^\top\hat{p}_k\!)(\!\nabla\! f(x_k) \!-\! g_k\!)^{\!\top}\! \hat{p}_k\\
  \leq & f(x_k) - \alpha_k\vert g_k^\top \hat{p}_k\vert - \frac{\alpha_k}{3}(\frac{17}{12}\varepsilon_h\alpha_k - \frac{L_h}{2}\alpha_k^2 - \varepsilon_g),
    \end{align*}
    where the first equality follows from  \(\|d_k\| = 1\) and the last inequality follows from \(\vert \hat{\lambda}_k\vert > \varepsilon_h\) and Assumption~\ref{assume:gkhkso}. If \(\alpha_k = \frac{17}{12L_h}\varepsilon_h\) and \(\varepsilon_h = \sqrt{L_h\varepsilon_g}\), then we have 
    \[
    \frac{\alpha_k}{3}(\frac{17}{12}\varepsilon_h\alpha_k - \frac{L_h}{2}\alpha_k^2 - \varepsilon_g) = \frac{17}{10368\sqrt{L_h}}\varepsilon_g^{3/2}.
\]
The conclusion follows.\footnote{The proof also shows that \(\alpha_k = \frac{\varepsilon_h}{L_h}\alpha\) (for some \(\alpha \in (\frac{4}{3}, \frac{3}{2})\)) is an alternative choice, which still guarantees a decrease of at least \(c\varepsilon_g^{3/2}\) for some constant \(c > 0\).}
\end{proof}

\begin{lemma}\label{lem:dfsolso}
Suppose Assumptions~\ref{assume:ncpso},~\ref{assume:gkhkso},  and~\ref{assume:lampso} are satisfied. If, at iteration \(k\) of Algorithm~\ref{alg:rnmso}, \(d_{\rm type} = {\rm SOL}\), then we have 
\begin{equation*}%\label{eq:dksolso}
    f(x_{k+1}) \leq f(x_k)  - \frac{5\sqrt{2}}{18\sqrt{3(L_h + \eta)}}\varepsilon_g^{ 3/2}.
    \end{equation*}
\end{lemma}
\begin{proof}
Notice that 
    \[
    Q_k + 2\varepsilon_hI_n \succeq (\lambda_{\min}(Q_k) + 2\varepsilon_h)I_n \succeq (\hat{\lambda}_k + 2\varepsilon_h - \frac{1}{2}\varepsilon_h)I_n \succeq \frac{1}{2}\varepsilon_h I_n,
    \]
    where the second inequality follows from Assumption~\ref{assume:lampso} and the last inequality follows from \(\hat{\lambda}_k \geq -\varepsilon_h\) when \(d_{\rm type} = {\rm SOL}\). 

    Define \(r_k = (Q_k + 2\varepsilon_hI_n)d_k + g_k\). For any \(\alpha\in[0, \frac{1}{2}]\), we have 
    \begin{align*}
        \alpha g_k^\top d_k + \frac{\alpha^2}{2}d_k^\top Q_kd_k =& \alpha (r_k - (Q_k + 2\varepsilon_hI_n)d_k)^\top d_k + \frac{\alpha^2}{2}d_k^\top Q_kd_k\\
        =&\alpha r_k^\top d_k -\alpha(1 - \frac{\alpha}{2})d_k^\top(Q_k + 2\varepsilon_hI_n)d_k - \alpha^2\varepsilon_h\|d_k\|^2\\
        \leq& -\frac{1}{2}\alpha(1 - \frac{\alpha}{2})\varepsilon_h\|d_k\|^2 \leq - \frac{3}{8}\alpha\varepsilon_h\|d_k\|^2,
    \end{align*}
where the first inequality follows from the property of the CG method, which guarantees \(r_k^\top d_k = 0\), and the second inequality follows from \(1 - \frac{\alpha}{2}\geq \frac{3}{4}\). Therefore, we have 
    \begin{align}
    f(x_k + \alpha d_k) \overset{\eqref{eq:uppercubic}}{\leq}& f(x_k) + \alpha\nabla f(x_k)^\top d_k + \frac{\alpha^2}{2}d_k^\top \nabla^2f(x_k)d_k + \frac{L_h\alpha^3}{6}\|d_k\|^3 \nonumber \\
    =& f(x_k) \!+\! \alpha g_k^\top d_k \!+\! \frac{\alpha^2}{2}\!d_k^\top Q_kd_k  \!+\! \alpha(\nabla f(x_k) \!-\! g_k)^\top d_k \!+\! \frac{\alpha^2}{2}\!d_k^\top (\nabla^2f(x_k) \!-\! Q_k)d_k \nonumber \\
    &+ \frac{L_h\alpha^3}{6}\!\|d_k\|^3 \nonumber \\
    \leq& f(x_k) \!-\! \frac{3\alpha}{8}\varepsilon_h\|d_k\|^2 \!+\! \frac{\alpha}{3}\varepsilon_g\|d_k\| \!+\! \frac{\alpha^2}{36}\varepsilon_h\|d_k\|^2 \!+\! \frac{L_h\alpha^3}{6}\|d_k\|^3 \nonumber \\
=& f(x_k) +\alpha\|d_k\|(- \frac{3}{8}\varepsilon_h\|d_k\| + \frac{1}{3}\varepsilon_g + \frac{\alpha}{36}\varepsilon_h\|d_k\| + \frac{L_h\alpha^2}{6}\|d_k\|^2). \label{eq:dfsol}   
    \end{align}
Notice that \(\|d_k\| > \frac{2\varepsilon_g}{\varepsilon_h}\) and \(\alpha_k\|d_k\| = \frac{\sqrt{2}\varepsilon_g^{1/2}}{\sqrt{3(L_h + \eta)}}\), we have 
\begin{align*}
- \frac{3}{8}\varepsilon_h\|d_k\| + \frac{1}{3}\varepsilon_g + \frac{\alpha_k}{36}\varepsilon_h\|d_k\| + \frac{L_h\alpha_k^2}{6}\|d_k\|^2 <& -\frac{5}{12}\varepsilon_g + \frac{\sqrt{L_h\varepsilon_g}}{36}\alpha_k\|d_k\| + \frac{L_h}{6}\alpha_k^2\|d_k\|^2\\
= &(- \frac{5}{12} + \frac{\sqrt{2L_h}}{36\sqrt{3(L_h + \eta)}} + \frac{L_h}{9(L_h+\eta)})\varepsilon_g\\
< & -\frac{5}{18}\varepsilon_g.
\end{align*}
Substituting this inequality into~\eqref{eq:dfsol}, we obtain the desired result. 
\end{proof}

One can observe that for \(d_{type} = {\rm SOL}\), the condition \(\|d_k\| > \frac{2\varepsilon_g}{\varepsilon_h}\) is essential for establishing the stated sufficient-decrease bound. 
Similar to~\cite[Lemma~8]{YXRWM23}, 
the following statement holds.  
\begin{lemma}\label{lem:ternatestep9}
Suppose that Assumptions~\ref{assume:ncpso},~\ref{assume:gkhkso},  and~\ref{assume:lampso} are satisfied. If  Algorithm~\ref{alg:rnmso} terminates at iteration \(k\) on line~\ref{line:9}, then we have 
\[
\|\nabla f(x_k + d_k)\| \leq (\frac{58}{9} + 2\hat{\mu})\varepsilon_g \quad {\rm and}\quad \lambda_{\min}(\nabla^2 f(x_k + d_k)) \geq -\frac{32}{9}\sqrt{L_h\varepsilon_g}.
\]
\end{lemma}
\begin{proof} 
Notice that
\begin{align*}
    \|\nabla f(x_k \!+\! d_k)\| \!\leq& \|\nabla f(x_k + d_k) \!-\! \nabla f(x_k) \!-\! \nabla^2f(x_k)d_k\| \!+\! \|\nabla f(x_k) \!-\! g_k\| \!+\! \|(\nabla^2f(x_k) \!-\! Q_k)d_k\| \\
    &+ \|(Q_k + 2\varepsilon_hI_n)d_k + g_k\| + 2\varepsilon_h\|d_k\|\\
    \!\leq\!& \frac{L_h}{2}\|d_k\|^2 + \frac{1}{3}\varepsilon_g + (\frac{37}{18} + \hat{\mu})\varepsilon_h\|d_k\| \leq (\frac{58}{9} + 2\hat{\mu})\varepsilon_g,
\end{align*}
where the second inequality follows from Assumptions~\ref{assume:ncpso} (ii) and~\ref{assume:gkhkso} and~\eqref{eq:dkso}, the last inequality follows from \(\|d_k\| \leq \frac{2\varepsilon_g}{\varepsilon_h}\).

Under Assumption~\ref{assume:ncpso}(ii), we have 
\[
\|\nabla^2f(x_k + d_k) - \nabla^2f(x_k)\| \leq L_h\|d_k\|, 
\]
which yields
\begin{align*}
    \lambda_{\min}(\nabla^2f(x_k + d_k)) \geq& \lambda_{\min}(\nabla^2f(x_k)) - L_h\|d_k\| \geq \!\lambda_{\min}(Q_k) \!-\! \frac{1}{18}\varepsilon_h \!-\! L_h\|d_k\|\\
    \geq& \hat{\lambda}_k - \frac{1}{2}\varepsilon_h- \frac{1}{18}\varepsilon_h - L_h\|d_k\| \geq -\frac{14}{9}\varepsilon_h - L_h\frac{2\varepsilon_g}{\varepsilon_h}    = -\frac{32}{9}\sqrt{L_h}\varepsilon_g^{1/2}, 
\end{align*}
where the second inequality follows from Assumption~\ref{assume:gkhkso}, the third inequality follows from Assumption~\ref{assume:lampso}, and the fourth inequality follows from \(\hat{\lambda}_k \geq -\varepsilon_h\). 
\end{proof}

\begin{theorem}\label{th:conrnmso}
    Suppose that Assumptions~\ref{assume:ncpso},~\ref{assume:gkhkso},  and~\ref{assume:lampso} are satisfied. For a given \(\varepsilon > 0\), let \(\varepsilon_g = \varepsilon\), \(\varepsilon_h = \sqrt{L_h\varepsilon}\), and define
    \[
    \overline{K}_3 := \left\lceil \frac{(f(x_0) - f_*)}{\min\{c_{nc}, c_{sol}\}}\varepsilon^{-3/2}\right\rceil + 1,
    \]
    where \(c_{sol} = \frac{5\sqrt{2}}{18\sqrt{3(L_h + \eta)}}\) and \(c_{nc} = \frac{17}{10368\sqrt{L_h}}\). Then Algorithm~\ref{alg:rnmso} terminates in at most \(\overline{K}_3\) iterations and returns a point \(x_{out}\) satisfying\footnote{`\(\lesssim\)' and `\(\gtrsim\)' denote that the corresponding inequality holds up to a certain constant that is independent of \(\varepsilon\) and \(L_h\).} \(\|\nabla f(x_{out})\| \lesssim \varepsilon\) and \(\lambda_{\min}(\nabla^2f(x_{out})) \gtrsim -\sqrt{L_h\varepsilon}\).  
\end{theorem}
\begin{proof}
From Lemmas~\ref{lem:dfncso} and~\ref{lem:dfsolso}, for any non‑terminating iteration \(k\), we have 
\[
f(x_{k+1}) \leq f(x_k) - \min\{c_{sol}, c_{nc}\}\varepsilon^{3/2}.
\]
Suppose, for contradiction, that all the first \(\overline{K}_3\) iterations are nonterminating. Summing the above inequality and using the definition of \(\overline{K}_3\), we obtain
\[
f(x_0) - f(x_{\overline{K}_3}) \geq \overline{K}_3\min\{c_{sol}, c_{nc}\}\varepsilon^{3/2} > f(x_0) - f^*. 
\]
Hence, Algorithm~\ref{alg:rnmso} terminates within at most \(\overline{K}_3\) iterations. It remains to verify the stationarity of the point returned at termination. If the algorithm terminates at line~\ref{line:17}, then \(\|g_k\| \leq \varepsilon_g\) and \(\hat{\lambda}_k \geq -\varepsilon_h\). Assumptions~\ref{assume:gkhkso} and~\ref{assume:lampso} therefore give  
\[
\|\nabla f(x_k)\| \leq \frac{4}{3}\varepsilon_g \quad {\rm and}\quad \lambda_{\min}(\nabla^2f(x_k)) \geq -\frac{14}{9}\varepsilon_h. 
\]
If the algorithm terminates at line~\ref{line:9}, the claimed bounds follow from Lemma~\ref{lem:ternatestep9}. 
\end{proof}

We now analyze the operation complexity of Algorithm~\ref{alg:rnmso} in terms of approximate gradient evaluations (\(g_k\)) and approximate Hessian-vector products (\(Q_kv\) for \(v\in\mathbb{R}^n\)). The latter mainly arises from the estimation of the minimum eigenvalue of \(Q_k\) and solving the regularized Newton system. 

\begin{corollary}\label{coro:ccrnmso}
Suppose that Assumptions~\ref{assume:ncpso} and~\ref{assume:gkhkso} are satisfied. At each iteration \(k\), suppose that the minimum eigenvalue of \(Q_k\) is estimated by the Lanczos method with a random initial vector, with failure probability \(\delta_k = \delta/((k+1)(k+2))\) for some \(0 < \delta \ll 1\), and that CG is used to solve the regularized Newton system in Algorithm~\ref{alg:rnmso}. For any \(\varepsilon > 0\), let \(\varepsilon_g = \varepsilon\), \(\varepsilon_h = \sqrt{L_h\varepsilon}\), and let \(\overline{K}_3\) be defined as in Theorem~\ref{th:conrnmso}. Then, with probability at least \(1 - \delta\), the algorithm terminates at a point \(x_{out}\) satisfying \(\|\nabla f(x_{out})\| \lesssim \varepsilon\) and \(\lambda_{\min}(\nabla^2f(x_{out})) \gtrsim -\sqrt{L_h\varepsilon}\) after at most
\[
\widetilde{\mathcal{O}}(\varepsilon^{-3/2}\min\{n, \varepsilon^{-1/4}\})
\]
approximate gradient evaluations and approximate Hessian-vector products in total.  
For \(n\) sufficiently large, with the same probability, the bound is \(\widetilde{\mathcal{O}}(\varepsilon^{-7/4})\).
\end{corollary}
\begin{proof}
At each iteration \(k\), by~\cite[Lemma 10]{ROW20}, with probability at least \(1 - \delta_k\), Assumption~\ref{assume:lampso} holds with at most 
\[
\min\{n, \mathcal{O}(\varepsilon_h^{-1/2}\ln(n/\delta_k))\} = \min\{n, \mathcal{O}(\varepsilon^{-1/4}\ln(n(\overline{K}_3 + 2)/\delta))\}
\]
matrix-vector multiplications. Since \(\sum_{k=0}^{\infty}\delta_k = \delta\), a union bound shows that, with probability at least \(1 - \delta\), Assumption~\ref{assume:lampso} holds at every Lanczos call made before termination. 

By~\cite[Lemma 1]{ROW20}, CG requires \(\min\{n, \widetilde{O}(\varepsilon_h^{-1/2})\} = \min\{n, \widetilde{O}(\varepsilon^{-1/4})\}\) 
matrix-vector multiplications. Hence, approximate Hessian-vector products of Algorithm~\ref{alg:rnmso} can be bounded by
\begin{align*}
&\sum_{k=0}^{\overline{K}_3 - 1}\min\{n, \widetilde{O}(\varepsilon^{-1/4})\} + \min\{n, \mathcal{O}(\varepsilon^{-1/4}\ln(n(\overline{K}_3 + 2)/\delta))\}\\
=& \mathcal{O}(\overline{K}_3(\min\{n, \widetilde{O}(\varepsilon^{-1/4})\} + \min\{n, \mathcal{O}(\varepsilon^{-1/4}\ln(n(\overline{K}_3 + 2)/\delta))\}))\\
=&\widetilde{\mathcal{O}}(\varepsilon^{-3/2}\min\{n, \varepsilon^{-1/4}\}).
\end{align*}
For \(n\) sufficiently large, the above bound simplifies to \(\widetilde{\mathcal{O}}(\varepsilon^{-7/4})\).  
\end{proof}

\begin{remark}[Strongly convex case]
If \(f\) is \(\sigma\)-strongly convex and Assumption~\ref{assume:gkhkso} holds with \(\varepsilon_h = \sqrt{L_h}\varepsilon_g^{1/2}\), then \(\varepsilon_g < \frac{18^2\sigma^2}{L_h}\) ensures \(Q_k \succeq 0\). Consequently, \(Q_k + 2\sqrt{L_h}\varepsilon_g^{1/2}I_n \succeq 2\sqrt{L_h}\varepsilon_g^{1/2}I_n\). 
Under this condition, negative curvature steps cannot occur, and the minimum-eigenvalue estimation can be omitted, yielding Algorithm~\ref{alg:offrn2cmsconv}. 

%%%%%%%%%%%%%%%%%%%%%%%%%%%%
%          remark: OFF-RNM-II
%%%%%%%%%%%%%%%%%%%%%%%%%%%%
\begin{algorithm*}[h!]
\caption{OFF Newton-type method for strongly convex optimization (OFF-Newton-II).}\label{alg:offrn2cmsconv}
\begin{algorithmic}[1]
\Require{\(x_0\), \(L_h\), \(\varepsilon_g\in(0, L_h)\), \(\hat{\mu} > 0\), and \(\eta > 0\).}
\For{\(k = 0, 1, \ldots\)}
\State{Compute \(g_{k}\) and \(Q_{k}\);}
\If{\(\|g_k\| \leq \varepsilon_g\)}
\State{Terminate and return \(x_k\);}
\EndIf
\State{Invoke the CG method to solve the regularized Newton system \((Q_k + 2\sqrt{L_h}\varepsilon_g^{1/2}I_n)d = -g_k\) with the zero vector as the initial guess, and return a vector \(d_k\) that satisfies}
\[
   \| (Q_k + 2\sqrt{L_h}\varepsilon_g^{1/2}I_n)d_k + g_k \| \leq \hat{\mu}\sqrt{L_h}\varepsilon_g^{1/2}\|d_k\|;
\]
\If{\(\|d_k\| \leq \frac{2}{\sqrt{L_h}}\varepsilon_g^{1/2}\)}
\State{Terminate and return \(x_k + d_k\);}\label{alg:line9}
\Else
\State{\(x_{k+1} \gets x_k + \frac{\sqrt{2}}{\sqrt{3(L_h+\eta)}\|d_k\|}\varepsilon_g^{1/2}d_k\);}
\EndIf
\EndFor\\
\Return{\(\{x_k\}\)}
\end{algorithmic}
\end{algorithm*}

If Assumption~\ref{assume:ncpso} holds, it follows immediately from Lemmas~\ref{lem:dfsolso} and~\ref{lem:ternatestep9} that Algorithm~\ref{alg:offrn2cmsconv} terminates in at most \(\lceil \frac{f(x_0) - f^*}{c_{\rm sol}}\varepsilon_g^{-3/2}\rceil + 1\) iterations and returns a point \(x_{out}\) satisfying \(\|\nabla f(x_{out})\| \lesssim \varepsilon\), where \(c_{\rm sol}\) is defined in Lemma~\ref{lem:dfsolso}. 

Both Algorithms~\ref{alg:rnmsc} and~\ref{alg:offrn2cmsconv} are OFF Newton-type methods for strongly convex optimization, but impose different derivative-accuracy requirements. Unlike Algorithm~\ref{alg:rnmsc}, whose fast local convergence requires an initial point sufficiently close to the minimizer, Algorithm~\ref{alg:offrn2cmsconv} provides a global iteration-complexity guarantee. 
\end{remark}

\subsection{Finite-sum optimization}\label{sec:OFFn2cpfs}
In this section, we study the complexity of Algorithm~\ref{alg:rnmso} for finite-sum problem~\eqref{eq:fs}. We retain the sampling scheme of Section~\ref{sec:offrnm3}. Let \(\mathcal{F}_k\) denote the information generated by \(x_0\) and all random choices made before sampling at iteration \(k\) and define \(\mathcal{G}_k = \sigma(\mathcal{F}_k, g_k, Q_k)\).

\begin{assumption}\label{assume:conrnmsofs}
The function \(f: \mathbb{R}^n \to \mathbb{R}\) is twice continuously differentiable and bounded below by a finite constant \(f_*\). For any \(x_0 \in \mathbb{R}^n\), there exists a constant $\bar{\epsilon}_0>0$ such that,
for each $i\in[m]$, $f_i$ is twice continuously differentiable
on \(\Omega_3
:= \{x\in\mathbb R^n \mid {\rm dist}\bigl(x,\mathcal L_f(x_0)\bigr)
<\bar{\epsilon}_0\}\). 
Moreover, $\nabla f_i$ and $\nabla^2 f_i$ are Lipschitz continuous
on $\Omega_3$ with constants
$L_{g_i}>0$ and $L_{h_i}>0$, respectively. 
\end{assumption}
Under Assumption~\ref{assume:conrnmsofs}, both \(\nabla f\) and \(\nabla^2f\) are Lipschitz continuous on $\Omega_3$ with constants \(L_g := \frac{1}{m}\sum_{i=1}^mL_{g_i}\) and \(L_h := \frac{1}{m}\sum_{i=1}^mL_{h_i}\), respectively.

\begin{theorem}\label{th:conrnmsofs2}
Suppose that Assumption~\ref{assume:conrnmsofs} is satisfied, and assume also that the component gradients and Hessians are uniformly bounded on \(\Omega_3\) by \(\widehat{U}_g :=\sup_{\substack{x\in\Omega_3\\ i\in[m]}} \|\nabla f_i(x)\| < +\infty\) and \(\widehat{U}_h :=\sup_{\substack{x\in\Omega_3\\ i\in[m]}}\|\nabla^2 f_i(x)\| < +\infty\). For a prescribed failure probability \(0 < \tilde{\delta} < 1\), let \(0 < p_{g,k}, p_{h,k}, p_{\lambda, k} < 1\) satisfy \(\sum_{k=0}^{\infty}p_{g,k} + p_{h,k} + p_{\lambda,k} \leq \tilde{\delta}\). For any \(0 < \varepsilon < \min\{L_h, 3\widehat{U}_g, \frac{L_h\bar{\epsilon}_0^2}{4}, 1\}\), let \(\varepsilon_g = \varepsilon\) and \(\varepsilon_h = \sqrt{L_h \varepsilon}\). At each iteration, construct the estimators \(g_k\) and \(Q_k\) by~\eqref{eq:subsample}, using the following sample sizes  
\[
\lvert \mathcal{S}_{g,k}\rvert = \min\{m, \lceil \frac{36\widehat{U}_g^2}{\varepsilon^2}\ln\frac{2(n+1)}{p_{g,k}}\rceil\}, \qquad \lvert \mathcal{S}_{h,k}\rvert = \min\{m, \lceil \frac{5184 \widehat{U}_h^2}{L_h\varepsilon}\ln\frac{2n}{p_{h,k}}\rceil\}.
\]
Suppose that
\begin{equation}\label{eq:assumelambda}
\mathbb{P}(\hat{\lambda}_k - \lambda_{\min}(Q_k) \leq \frac{\varepsilon_h}{2} \,\big|\, \mathcal{G}_k) \geq 1 - p_{\lambda, k}.
\end{equation}
Let $\overline{K}_3$ be defined in Theorem~\ref{th:conrnmso}. Then, with probability at least $1-\tilde{\delta}$, Algorithm~\ref{alg:rnmso} terminates in at most $\overline{K}_3$ iterations and returns a point $x_{out}$ satisfying
\[
\|\nabla f(x_{out})\| \le (\frac{58}{9} + 2\hat{\mu})\varepsilon \qquad {\rm and}\qquad \lambda_{\min}\big(\nabla^2 f(x_{out})\big) \ge -\frac{32}{9}\sqrt{L_h \varepsilon}. 
\]
\end{theorem}
\begin{proof}
Define the joint accuracy event
\[
E_k =\{\|g_k - \nabla f(x_k)\| \le \dfrac{\varepsilon_g}{3},~\|Q_k - \nabla^2 f(x_k)\| \le \dfrac{\varepsilon_h}{18},~\hat{\lambda}_k - \lambda_{\min}(Q_k) \le \dfrac{\varepsilon_h}{2}\}. 
\]
Let \(B_k = \{\text{iteration}~k~\text{is~reached}\} \cap (\cap_{j=0}^{k-1}E_j)\). Then \(B_k \in \mathcal{F}_k\). 

We first prove by induction that \(x_k \in \mathcal{L}_f(x_0)\) on \(B_k\). The assertion holds at \(x_0\). Suppose it holds at iteration \(k \geq 0\). We have 
\[
\|x_{k+1} - x_k\| \leq \max\{\frac{2\varepsilon_g}{\varepsilon_h}, \frac{17}{12L_h}\varepsilon_h, \frac{\sqrt{2}}{\sqrt{3(L_h + \eta)}}\varepsilon_g^{1/2}\} \leq 2\sqrt{\varepsilon / L_h} < \bar{\epsilon}_0
\]
and hence \([x_k, x_{k+1}] \subseteq \Omega_3\). Assumption~\ref{assume:conrnmsofs} therefore supplies the smoothness bounds used in Lemmas~\ref{lem:dfncso}-\ref{lem:ternatestep9} along the update segment. For a nonterminating update, Lemmas~\ref{lem:dfncso} and~\ref{lem:dfsolso} give
\[
f(x_{k+1}) \leq f(x_k) - \min\{c_{\mathrm{nc}},\, c_{\mathrm{sol}}\}\epsilon^{3/2} \leq f(x_0).
\]
Thus, \(x_{k+1} \in \mathcal{L}_f(x_0) \subseteq \Omega_3\) on \(B_{k+1}\). This completes the induction. 

On the event \(B_k\), using Lemma~\ref{lem:subsampling}, the prescribed sample size, and \(\varepsilon < 3\widehat{U}_g\), we obtain 
\[
\mathbb{P}(\|g_k - \nabla f(x_k)\| \leq  \frac{\varepsilon_g}{3} \mid \mathcal{F}_k) \geq 1 - p_{g, k}, \quad \mathbb{P}(\|Q_k - \nabla^2 f(x_k)\| \leq \frac{\epsilon_h}{18} \mid  \mathcal{F}_k) \geq 1 - p_{h,k}.
\]
By~\eqref{eq:assumelambda}, we have 
\[
\mathbb{P}(\hat{\lambda}_k - \lambda_{\min}(Q_k) > \frac{\varepsilon_h}{2} \mid \mathcal{F}_k) \leq p_{\lambda, k}. 
\]
The conditional union bound implies \(\mathbb{P}\bigl(B_k \cap E_k^c \bigr) \le p_{g,k} + p_{h,k} + p_{\lambda, k}\). 
Let \(A = \big(\cup_{k=0}^{\overline{K}_3 - 1}(B_k \cap E_k^c)\big)^c\). Then 
\[
\mathbb{P}(A) \geq 1 - \sum_{k=0}^{\overline{K}_3 - 1}\mathbb{P}(B_k \cap E_k^c) \geq 1 - \sum_{k=0}^{\infty}(p_{g,k} + p_{h,k} + p_{\lambda, k}) \geq 1 -\tilde{\delta}.
\]

For the remainder of the proof, we work on \(A\). The preceding induction ensures that all iterations before termination are well-defined. If the first $\overline{K}_3$ iterations were all nonterminating, summing the decrease inequalities gives 
\[
f(x_0) - f(x_{\overline{K}_3}) \ge \overline{K}_3 \min\{c_{\mathrm{nc}},\, c_{\mathrm{sol}}\} \varepsilon^{3/2} > f(x_0) - f^*,
\]
contradicting $f(x_{\overline{K}_3})\ge f^*$. Thus the algorithm terminates within at most \(\overline{K}_3\) iterations.

If termination occurs at iteration \(k\) on line~\ref{line:17}, the accuracy conditions give
\[
\|\nabla f(x_k)\| \le \frac{4}{3}\varepsilon \qquad {\rm and}\qquad \lambda_{\min}\bigl(\nabla^2 f(x_k)\bigr) \ge -\frac{14}{9}\sqrt{L_h\varepsilon}.
\]
If termination occurs at iteration \(k\) on line~\ref{line:9}, we have \(x_k + d_k \in \Omega_3\) since \(\|d_k\| \leq 2L_h^{-\frac{1}{2}}\varepsilon^{\frac{1}{2}} < \bar{\epsilon}_0\).  Lemma~\ref{lem:ternatestep9} gives
\[
\|\nabla f(x_k + d_k)\| \le \left(\frac{58}{9} + 2\hat{\mu}\right)\varepsilon \qquad {\rm and}\qquad \lambda_{\min}\bigl(\nabla^2 f(x_k + d_k)\bigr) \ge -\frac{32}{9}\sqrt{L_h\varepsilon}.
\]
Both termination cases satisfy the stated bounds. Since these conclusions hold on $A$, whose probability is at least $1-\tilde{\delta}$, the proof is complete.
\end{proof}

We next characterize the sample-operation complexity of  Algorithm~\ref{alg:rnmso} for the finite-sum problem~\eqref{eq:fs}, which counts either a component gradient evaluation (\(\nabla f_i(x)\)) or a component Hessian-vector product (\(\nabla^2f_i(x)p\) for some \(p\in\mathbb{R}^n\)). 

 \begin{corollary}\label{coro:ccrnmsofs2}
Suppose that the assumptions of Theorem~\ref{th:conrnmsofs2} hold.   
For a prescribed failure probability \(0 < \tilde{\delta} < 1\), choose \(p_{g,k} =  p_{h,k} = p_{\lambda, k} = \frac{\tilde{\delta}}{3(k+1)(k+2)}\), \(k \geq 0\). For any \(0 < \varepsilon < \min\{L_h, 3\widehat{U}_g, \frac{L_h\bar{\epsilon}_0^2}{4}, 1\}\), let \(\varepsilon_g = \varepsilon\) and \(\varepsilon_h = \sqrt{L_h \varepsilon}\). 
Use the sampling rules specified in Theorem~\ref{th:conrnmsofs2} with these tolerances and failure probabilities.  
Suppose that, at each iteration, the randomly initialized Lanczos method is used to estimate the minimum eigenvalue of the approximate Hessian with the accuracy and failure probability specified  in~\eqref{eq:assumelambda}, and that CG is used to solve the regularized Newton system in Algorithm~\ref{alg:rnmso}. Then, on the event established in Theorem~\ref{th:conrnmsofs2}, which has probability at least \(1 - \tilde{\delta}\), the algorithm terminates at a point \(x_{out}\) satisfying \(\|\nabla f(x_{out})\| \lesssim \varepsilon\) and \(\lambda_{\min}(\nabla^2f(x_{out})) \gtrsim -\sqrt{L_h\varepsilon}\) after at most
\[
\widetilde{\mathcal{O}}(\varepsilon^{-3/2}\min\{m, \varepsilon^{-2}\})
\]
component gradient evaluations, 
\[
\widetilde{\mathcal{O}}(\varepsilon^{-3/2}\min\{m, \varepsilon^{-1}\}\min\{n, \varepsilon^{-1/4}\})
\]
component Hessian-vector products, and  
\[
\widetilde{\mathcal{O}}(\varepsilon^{-3/2}(\min\{m, \varepsilon^{-2}\} + \min\{m, \varepsilon^{-1}\}\min\{n, \varepsilon^{-1/4}\}))
\]
total sample operations. For \(m\) and \(n\) sufficiently large, with the same probability, the bound is \(\widetilde{\mathcal{O}}(\varepsilon^{-7/2})\). 
\end{corollary}
\begin{proof}
Throughout the proof, we work on the event established in Theorem~\ref{th:conrnmsofs2}, which has probability at least \(1 - \tilde{\delta}\) and on which Algorithm~\ref{alg:rnmso} terminates within \(\overline{K}_3\) iterations.

Since \(\ln\frac{1}{p_{g,k}} = \mathcal{O}(\ln\frac{\overline{K}_3+2}{\tilde{\delta}})\), we have 
\[
\vert \mathcal{S}_{g, k}\vert = \min\{m, \mathcal{O}(\varepsilon^{-2}\ln\frac{n(\overline{K}_3 + 2)}{\tilde{\delta}})\}, \qquad \vert \mathcal{S}_{h, k}\vert = \min\{m, \mathcal{O}(\varepsilon^{-1})\ln\frac{n(\overline{K}_3 + 2)}{\tilde{\delta}})\}. 
\]
Recall that \(\overline{K}_3 = \mathcal{O}(\varepsilon^{-3/2})\), the total number of component gradient evaluations is  
\[
N_g = \mathcal{O}(\overline{K}_3\vert \mathcal{S}_{g, k}\vert) = \widetilde{\mathcal{O}}(\varepsilon^{-3/2}\min\{m, \varepsilon^{-2}\}). 
\]
Since the Lanczos method is adopted to estimate the minimum eigenvalue of \(Q_k\), by~\cite[Lemma 10]{ROW20},~\eqref{eq:assumelambda} holds requiring at most 
\[
\min\{n, \mathcal{O}(\varepsilon_h^{-1/2}\ln(n/p_{\lambda,k}))\} = \min\{n, \mathcal{O}(\varepsilon^{-1/4}\ln\frac{n(\overline{K}_3 + 2)}{\tilde{\delta}})\}
\]
matrix-vector multiplications. By~\cite[Lemma 1]{ROW20}, CG requires \(\min\{n, \widetilde{\mathcal{O}}(\varepsilon^{-1/4})\}\) 
matrix-vector multiplications. Hence, the total number of component Hessian-vector products is 
\begin{align*}
N_h =& \mathcal{O}(\overline{K}_3\vert \mathcal{S}_{h, k}\vert(\min\{n, \widetilde{\mathcal{O}}(\varepsilon^{-1/4})\} + \min\{n, \mathcal{O}(\varepsilon^{-1/4}\ln\frac{n(\overline{K}_3 + 2)}{\tilde{\delta}})\})) \\
=& \widetilde{\mathcal{O}}(\varepsilon^{-3/2}\min\{m, \varepsilon^{-1}\}\min\{n, \varepsilon^{-1/4}\}). 
\end{align*}
Therefore, 
\[
N_g + N_h = \widetilde{\mathcal{O}}(\varepsilon^{-3/2}(\min\{m, \varepsilon^{-2}\} + \min\{m, \varepsilon^{-1}\}\min\{n, \varepsilon^{-1/4}\})). 
\]
For sufficiently large \(m\) and \(n\), the above bound simplifies to 
\[
\widetilde{\mathcal{O}}(\varepsilon^{-3/2}(\varepsilon^{-2} + \varepsilon^{-5/4})) = \widetilde{\mathcal{O}}(\varepsilon^{-7/2}). 
\]
\end{proof}

%%%%%%%%%%%%%%%%%%%%%%%%%%%%
%.    Conclusions
%%%%%%%%%%%%%%%%%%%%%%%%%%%%

\section{Conclusions}\label{sec:con}

We develop OFF Newton-type methods for first- and second-order stationarity with inexact derivatives. Explicit step-size rules, derivative-accuracy conditions, and linear-solver residual bounds yield convergence and complexity guarantees without objective-function evaluations. Subsampled variants extend these results to finite-sum problems. The methods require prior knowledge of problem-dependent constants, including smoothness constants and, where applicable, the strong-convexity modulus. Developing adaptive OFF schemes that relax these prior-knowledge requirements while preserving convergence guarantees is left for future research. 

\section*{Funding}
This work is supported by funds from National Natural Science Foundation of China (\# 12271217). 

%%===========================================================================================%%
\begin{appendices}

\section{Methods using inexact derivatives}\label{appendix:table1}
Table~\ref{tab:summethods} summarizes methods using inexact derivatives along three dimensions: the derivative- accuracy conditions, convergence guarantees, and function-evaluation requirements. 

\begin{sidewaystable*}[p] 
{\footnotesize
\caption{Summary of methods using inexact derivatives. ``\(\mathbb{E}\)" and ``\(\mathbb{P}\)" refer to ``in expectation" and ``in probability", respectively. ``2o-p" stands for second-order stationary point. Parameters indexed by \(k\) vary with iterations. ``conv." stands for ``convex" and ``s. conv." stands for ``\(\sigma\)-strongly convex". \(\mathcal{O}(\cdot)\) stands for the standard big-O asymptotic upper bound, which hides absolute constant factors and \(\widetilde{\mathcal{O}}(\cdot)\) hides polylogarithmic factors. The probabilistic inexact criterion for~\eqref{eq:fs} is typically formulated under the probabilistic convergence conditions of subsampling. ``\(\circ\)" means function-value estimates is required. ``L.c." denotes whether Lipschitz constants are required. \(\mathcal{F}_k\) denotes the information available before sampling at iteration \(k\).}\label{tab:summethods}
%\hspace*{-1.5cm} %
\setlength{\tabcolsep}{1pt}
\begin{tabular*}{\textwidth}{@{\extracolsep{\fill}}c|c|ccccc}\hline\hline
{Algs.} &   {Probs.}     &{Approx. gradient (\(g_k\))} & {Approx. Hessian (\(Q_k\))} & rate/iter.comp.  & OFF          & L.c.                                                                                                                                                   \\ \hline\hline  %---0
{ASGRAD~\cite{GJT25b}} & {\eqref{eq:ucfun}} & \(\mathbb{E}[g_k \mid \mathcal{F}_k] = \nabla f(x_k)\), \(\|g_k\|_{\infty} \leq \kappa_g\) & {-} & {\(\widetilde{\mathcal{O}}(1/k^{1/4})\) (\(\mathbb{E}\))} & {\(\checkmark\)} &\(\times\)\\   \hdashline %---3
SSN~\cite{RM19} &  \eqref{eq:fs} (s. conv.)&  \(\mathbb{P}(\|g_k \!-\! \nabla f(x_k)\|\!\leq\! \rho^k\varepsilon_2) \!\geq\! 1 \!-\! \bar{\delta}\)  &  \(\mathbb{P}(\|Q_k - \nabla^2 f(x_k)\| \leq \sigma\varepsilon_1) \geq 1 - \bar{\delta}\) &  - & \(\times\)  & \(\times\)\\\hline\hline  %---4
{SKOFFARp\footnote{This is a random subspace method that requires exact gradients; only the results for \(p = 2\) are listed.}~\cite{BGMT25}} & {\eqref{eq:ucfun}} &  \(\nabla f(x_k)\) & \(Q_k \succeq 0\), \(\|Q_k\| \leq \kappa_Q\) & \(\mathcal{O}(\varepsilon^{-3/2})\) (\(\mathbb{P}\))  & {\(\checkmark\)} & \(\times\) \\ \hdashline %---5
  \multirow{2}{*}{StOFFARp~\cite{GJT25c}} & \multirow{2}{*}{\eqref{eq:ucfun}/ \eqref{eq:fs}}  & \(\mathbb{E}[\|g_k \!-\! \nabla f(x_k)\|^2 \!\mid\! \mathcal{F}_k] \!\leq\! \kappa_D\!\!\sum_{i=1}^M\|S_{k\!-\!i}\|^2\)\footnote{\(\|S_{k-i}\|\) denotes the history step size.} & - & \(\mathcal{O}(\varepsilon^{-2})\) (\(\mathbb{E}\)) & \multirow{2}{*}{\(\checkmark\)} & \multirow{2}{*}{\(\times\)}\\   %---6
 & &  \(\mathbb{E}[\|g_k \!-\! \nabla f(x_k)\|^\frac{3}{2} \!\mid\! \mathcal{F}_k] \!\leq\! \kappa_D\!\!\sum_{i=1}^M\|S_{k\!-\!i}\|^3\) & \(\mathbb{E}[\|Q_k \!-\! \nabla^2 f(x_k)\|^3 \!\mid\! \mathcal{F}_k] \!\leq\! \kappa_D\!\!\sum_{i=1}^M\|S_{k\!-\!i}\|^3\) & \(\mathcal{O}(\varepsilon^{-3/2})\) (\(\mathbb{E}\))  \\    \hline %6
 
ARC~\cite{CGT11a} (2o-p) &   {\eqref{eq:ucfun}} & \(\nabla f(x_k)\) & \(\|(Q_k - \nabla^2f(x_k))[s_k]\| \leq c\|s_k\|^2\) & \(\mathcal{O}(\varepsilon^{-3/2})\) & \(\times\) & \(\times\)  \\ \hdashline %---7

 TR~\cite{XRM20} (2o-p) &\multirow{2}{*}{\eqref{eq:ucfun}} &\multirow{4}{*}{\(\nabla f(x_k)\)} &\(\|(Q_k \!-\! \nabla^2f(x_k))s_k\| \!\leq\! \max\{\varepsilon_0, \Delta_k\}\|s_k\|\) & \(\mathcal{O}(\varepsilon^{-5/2})\) &\multirow{4}{*}{\(\times\)}  & \multirow{4}{*}{\(\times\)} \\ %---8 TR
 ARC~\cite{XRM20} (2o-p) & & &\(\|(Q_k - \nabla^2f(x_k))s_k\| \leq \delta^h\|s_k\|\) & \(\mathcal{O}(\varepsilon^{-3/2})\) &  \\  % -- 8 ARC
  TR~\cite{XRM20} (2o-p) &\multirow{2}{*}{\eqref{eq:fs}} &  &\(\mathbb{P}(\|Q_k \!-\! \nabla^2f(x_k)\|\!\leq\! \max\{\varepsilon_0, \Delta_k\}) \!\geq\! 1 \!-\! \bar{\delta}\) & \(\mathcal{O}(\varepsilon^{-5/2})\)  (\(\mathbb{P}\))&  \\  %---8 TR
ARC~\cite{XRM20} (2o-p) & & &\(\mathbb{P}(\|Q_k - \nabla^2f(x_k)\|\leq \delta^h) \geq 1 - \bar{\delta}\) & \(\mathcal{O}(\varepsilon^{-3/2})\)   (\(\mathbb{P}\))&  \\    \hdashline % -- 8 ARC

I-TR~\cite{YXRM21} (2o-p) & \multirow{2}{*}{\eqref{eq:ucfun}} & \(\|g_k - \nabla f(x_k)\| \leq \frac{(1-\eta)\varepsilon}{4}\) & \(\|Q_k - \nabla^2f(x_k)\| \leq \min\{\frac{(1-\eta)\nu\epsilon^{1/2}}{2},1\}\) & \(\mathcal{O}(\varepsilon^{-3/2})\) & \multirow{2}{*}{\(\times \)}  &  \multirow{2}{*}{\(\checkmark\)} \\ 
I-ARC~\cite{YXRM21} (2o-p) & & \(\|g_k - \nabla f(x_k)\| \leq \delta^g(\varepsilon, L)\)\footnote{\(\delta^g(\varepsilon, L) = \frac{1-\eta}{192L_h}(\sqrt{K_H^2 + 8L_h\varepsilon} - K_H)^2\), \(\delta^h(\varepsilon, L) = \frac{1-\eta}{6}\min\{\frac{1}{4}(\sqrt{K_H^2 + 8L_h\varepsilon} - K_H),\nu\varepsilon^{1/2}\}\), where \(K_H \leq L_g + \delta^h(\varepsilon, L)\). } & \(\|Q_k - \nabla^2f(x_k)\| \leq \delta^h(\varepsilon, L)\) &  \(\mathcal{O}(\varepsilon^{-3/2})\) &   \\ \hdashline % -- 9
 
{TRqNE~\cite{BGMT22b}} (2o-p) & {\eqref{eq:ucfun}}/ \eqref{eq:fs} & \multicolumn{2}{c}{\(\mathbb{P}((\cap_{j=1}^{J_k}\mathcal{M}_{k,j}^{(1)}) \cap \mathcal{M}_{k}^{(2)} \mid \mathcal{F}_k) \geq \alpha_* > 1/2 \)}\footnote{\(\mathcal{M}_{k,j}^{(1)}\) (\(j = 1, \cdots, J_k\)) bounds the maximum reduction of the exact \(j\)-th-order Taylor model within the optimality ball by a fixed multiple of the approximate Taylor-model reduction at the computed displacement. \(\mathcal{M}_k^{(2)}\) imposes a relative error bound between the exact and approximate Taylor-model reductions at the trial step.} & \(\mathcal{O}(\varepsilon^{-3})\) (\(\mathbb{E}\)) & {\(\circ\)} & \(\times\) \\ \hdashline %--- 10

 ALAS~\cite{BDKKR22} (2o-p) & \eqref{eq:fs} & \multicolumn{2}{c}{\(\mathbb{P}(\max_{x\in[x_k, x_{k+1}]}\|g(x; \mathcal{S}_k) - \nabla f(x)\| \leq \kappa_g\varepsilon, \|Q_k - \nabla^2f(x_k)\| \leq \kappa_h\varepsilon^{1/2}\mid \mathcal{F}_k) \geq 1 - \bar{\delta}\)}\footnote{\(\max_{x\in[x_k, x_{k+1}]}\|g(x; \mathcal{S}_k) - \nabla f(x)\| = \max\{\|g_k - \nabla f(x_k)\|, \frac{1}{\vert\mathcal{S}_k\vert}\sum_{i\in\mathcal{S}_k}\nabla f_i(x_{k+1}) - \nabla f(x_{k+1}\|\}\)} & \(\mathcal{O}(\varepsilon^{-3/2})\) (\(\mathbb{E}\)) & \(\circ\) & \(\times\)\\  \hdashline %---11
 
 STR~\cite{LLHT23} (2o-p)&  \multirow{2}{*}{\eqref{eq:fs}} & \multirow{2}{*}{\(\|g_k - \nabla f(x_k)\| \leq \delta_g\)} & \multirow{2}{*}{\(\|Q_k - \nabla^2f(x_k)\| \leq v_0\delta_h\)} & \(\mathcal{O}(\max\{\hat{\varepsilon}^{1/2}\tilde{\varepsilon}^{-2}, \hat{\varepsilon}^{-3/2}\})\)\footnote{\(\hat{\varepsilon} = \varepsilon^{1/2} - \delta_h\), \(\tilde{\varepsilon} = \varepsilon - \delta_g\).} & \multirow{2}{*}{\(\circ\)} & \multirow{2}{*}{\(\times\)} \\ 
 SARC~\cite{LLHT23} (2o-p) & & & & \(\mathcal{O}(\max\{\tilde{\varepsilon}^{-3/2}, \hat{\varepsilon}^{-3/2}\})\) &  \\  \hdashline  %---12
 
\multirow{2}{*}{I-N-CG~\cite{YXRWM23} (2o-p)} & {\eqref{eq:ucfun}} & \(\|g_k - \nabla f(x_k)\| \leq \delta_{g,k}(\varepsilon,L_h,d,g)\)\footnote{\(\delta_{g,k}(\varepsilon,L_h,d,g) = \frac{1-\zeta}{8}\min\{\frac{3\varepsilon}{65(L_h + \eta)}, \max\{\varepsilon, \min\{\varepsilon^{1/2}\|d_k\|, \|g_k\|, \|g_{k+1}\|\}\}\}\), \(\delta_{h,k}(\varepsilon, L_h) = \frac{(1-\zeta)}{4}\sqrt{L_h}\varepsilon^{1/2}\), \(\delta_{g,k}(\varepsilon,L_h) = \frac{1-\zeta}{8}\min\{\frac{3L_h\varepsilon}{65(L_h+\eta)}, \varepsilon\}\). For Problem~\eqref{eq:ucfun}, the probabilistic guarantee arises from the minimum-eigenvalue oracle.} & \(\|Q_k - \nabla^2f(x_k)\| \leq \delta_{h,k}(\varepsilon, L_h)\) & \multirow{2}{*}{\(\mathcal{O}(\varepsilon^{-3/2})\) (\(\mathbb{P}\))} & \multirow{2}{*}{\(\checkmark\)} & \multirow{2}{*}{\(\checkmark\)}\\
 & {\eqref{eq:fs}} & \(\mathbb{P}(\|g_k \!-\! \nabla f(x_k)\| \!\leq\! \delta_{g,k}(\varepsilon,L_h)) \geq 1 \!- \!\bar{\delta}\) & \(\mathbb{P}(\|Q_k \!\!-\!\! \nabla^2f(x_k)\| \!\!\leq\! \!\delta_{h,k}(\varepsilon, L_h)) \geq 1 \!\!-\!\! \bar{\delta}\) & &  \\ \hdashline %---13
 
 \multirow{2}{*}{ RA~\cite{LW25} (2o-p)} & \eqref{eq:ucfun} & \(\|g_k - \nabla f(x_k)\| \leq \frac{1}{3}\max\{\varepsilon, \|g_k\|\}\) & \(\|Q_k - \nabla^2f(x_k)\| \leq \frac{2}{9}\varepsilon^{1/2}\) & \(\mathcal{O}(\varepsilon^{-2})\) (\(\mathbb{E}\))\footnote{The expectation arises from the random sign introduced in the negative-curvature iteration steps.} & \multirow{2}{*}{\(\checkmark\)} & \multirow{2}{*}{\(\checkmark\)} \\ %---14
  & \eqref{eq:fs} & \(\mathbb{P}(\|g_k \!-\! \nabla f(x_k)\| \!\!\leq\!\! \frac{1}{3}\max\{\varepsilon, \|g_k\|\} \!\mid\! \mathcal{F}_k) \!\!\geq\!\! 1\!\! -\!\! \bar{\delta}\) & \(\mathbb{P}(\|Q_k - \nabla^2f(x_k)\| \leq \frac{2}{9}\varepsilon^{1/2} \!\mid\! \mathcal{F}_k) \geq 1 - \bar{\delta}\) & \(\mathcal{O}(\varepsilon^{-2})\) (\(\mathbb{P}\)) & 
 \\ \hline\hline %---14
\end{tabular*}
%\hspace*{-1.5cm} 
}
\end{sidewaystable*}

\end{appendices}
                     %%
%%===========================================================================================%%

\bibliographystyle{unsrt}

\bibliography{manuscript}

\end{document}